\documentclass[a4paper]{amsart}
\usepackage{amsmath}
\usepackage{amsfonts}
\usepackage{amssymb}
\usepackage{amsthm}
\usepackage{amsbsy}
\usepackage{enumerate}
\usepackage[shortlabels]{enumitem}
\usepackage[hidelinks=true]{hyperref}
\newcommand{\deriv}[2][]{\frac{d #1}{d #2}}
\newcommand{\pderiv}[2][]{\frac{\partial #1}{\partial #2}}
\newcommand{\horizontal}{horizontal }
\newcommand{\pathlag}[1]{\mathbb{S}^{#1}}
\newcommand{\pospathlag}[1]{\mathbb{S}^#1_+}

\DeclareMathOperator{\real}{Re}
\DeclareMathOperator{\imaginary}{Im}
\DeclareMathOperator{\vol}{vol}

\theoremstyle{plain}
\newtheorem{thm}{Theorem}[section]
\newtheorem{lemma}[thm]{Lemma}
\newtheorem{prop}[thm]{Proposition}
\newtheorem{cor}[thm]{Corollary}

\theoremstyle{definition}
\newtheorem{dfn}[thm]{Definition}

\theoremstyle{remark}
\newtheorem{rem}[thm]{Remark}

\begin{document}

\begin{abstract}
The space of positive Lagrangians in an almost Calabi-Yau manifold is an open set in the space of all Lagrangian submanifolds. A Hamiltonian isotopy class of positive Lagrangians admits a natural Riemannian metric $\Upsilon$, which gives rise to a notion of geodesics. We study geodesics of positive $O_n(\mathbb{R})$ invariant Lagrangian spheres in $n$-dimensional $A_m$ Milnor fibers. We show the existence and uniqueness of smooth solutions to the initial value problem and the boundary value problem. In particular, we obtain examples of smooth geodesics of positive Lagrangians in arbitrary dimension. As an application, we show that the Riemannian metric $\Upsilon$ induces a metric space structure on the space of positive $O_n(\mathbb{R})$ invariant Lagrangian spheres in the above mentioned Milnor fibers.
\end{abstract}

\title{Geodesics of positive Lagrangians in Milnor fibers}
\author{Jake P. Solomon and Amitai M. Yuval}
\date{Jan. 2016}
\maketitle

\section{Introduction}

Let $(X,\omega)$ be a symplectic manifold of dimension $2n$, and let $L$ be an $n$-dimensional smooth manifold.
Let $\mathcal{L}$ be the space of Lagrangian submanifolds of $X$ that are diffeomorphic to $L$. We think of $\mathcal{L}$ as an infinite dimensional manifold. Assume $L$ is compact and $X$ is almost Calabi-Yau. Namely, $X$ is a K\"{a}hler manifold equipped with a nonvanishing holomorphic form $\Omega$ of type $(n,0)$. We call a Lagrangian submanifold $\Lambda\in\mathcal{L}$  \emph{positive} if $\real\Omega|_\Lambda$ does not vanish anywhere, and we denote by $\mathcal{L}^+\subset\mathcal{L}$ the open subset of all positive Lagrangians. Let $\mathcal{O}\subset\mathcal{L}^+$ be a Hamiltonian isotopy class. For $\Lambda \in \mathcal{O},$ the tangent space $T_\Lambda\mathcal{O}$ is canonically isomorphic to the space
\[
\left\{h\in C^\infty(\Lambda)\left|\int_\Lambda h\real\Omega=0\right\}\right..
\]
Consider the Riemannian metric $\Upsilon$ on $\mathcal{O}$ given by
\[
\Upsilon(h,k) = \int_\Lambda hk\real\Omega,\quad h,k\in T_\Lambda\mathcal{O}.
\]
It is shown in \cite{solomon} and \cite{solomon-curv} that $\Upsilon$ gives rise to a Levi-Civita connection and thus a corresponding notion of geodesics. The geodesic equation can be seen to be a fully non-linear degenerate elliptic partial differential equation~\cite{rubinstein-solomon}. In the present work, we study the geodesic equation in a family of examples with a high degree of symmetry.

Let $G$ be a Lie group that acts on $X$ preserving $\omega$ and $\Omega.$ So $G$ also acts on $\mathcal{L}^+$, and we denote by $\mathcal{L}^+_G \subset \mathcal{L}^+$ the locus of $G$-invariant positive Lagrangians. Let $\mathcal{O} \subset \mathcal{L}^+_G$ be a Hamiltonian isotopy class of $G$-invariant Lagrangians and let $\widehat{\mathcal{O}} \subset \mathcal{L}^+$ be the Hamiltonian isotopy class containing $\mathcal{O}.$ By analogy with finite dimensional manifolds, it is natural to look for geodesics of $\widehat{\mathcal{O}}$ that are contained in $\mathcal{O}.$ Indeed, the fixed points of a Lie group acting by isometries on a finite dimensional Riemannian manifold constitute a totally geodesic submanifold.

This paper shows existence and uniqueness of geodesics when $X$ is an $n$-dimensional $A_m$ Milnor fiber and $\mathcal{O}$ is a Hamiltonian isotopy class of $O_n(\mathbb{R})$-invariant Lagrangian spheres called matching cycles. The study of such Lagrangian spheres was pioneered by~\cite{khovanov-seidel,shapere-vafa} and~\cite[Section~6]{thomas-yau}. We recall the definitions. Fix a positive integer $n$ and fix a complex polynomial $f$ in one variable, of degree $m+1$, with only simple zeros. The corresponding Milnor fiber is the algebraic variety
\begin{align*}
M^n=\{(z_1,\ldots,z_n,\zeta)\in\mathbb{C}^{n+1}|z_1^2+\ldots+z_n^2=f(\zeta)\}.
\end{align*}
In particular, $M^n$ is an almost Calabi-Yau manifold of complex dimension~$n$. See Section~\ref{milnorfiber} for more detail. The group $O_n(\mathbb{R})$ acts on $M^n$ via the $z$ coordinates. A matching cycle~\footnote{The definition given here is equivalent to the definition found in the literature, which is recalled in Section~\ref{milnorfiber}. In Seidel's terminology~\cite{seidel}, these spheres would be called naive matching cycles.} in $M^n$ is an $O_n(\mathbb{R})$ invariant Lagrangian sphere on which $O_n(\mathbb{R})$ does not act freely (this is always the case if $n \geq 2$). Denote by $\pathlag{n}$ the space of matching cycles, and denote by $\pospathlag{n}\subset\pathlag{n}$ the subspace consisting of positive Lagrangians.  It is shown in Section~\ref{milnorfiber} that $\pathlag{n}$ and $\pospathlag{n}$ are infinite dimensional manifolds. Our main results are the following.
\begin{thm}
\label{initial}
Let $\Lambda_0\in\pospathlag{n},\;\tau\in T_{\Lambda_0}\pathlag{n}$. Then for some $\epsilon>0$ there exists a geodesic $\Lambda:[0,\epsilon]\to\pospathlag{n}$ such that
\begin{align*}
\Lambda(0)=\Lambda_0,\quad\left.\deriv{t}\right|_{t=0}\Lambda(t)=\tau.
\end{align*}
For a fixed $\epsilon$, the geodesic $\Lambda$ is unique.
\end{thm}
\begin{thm}
\label{boundary}
Let $n\geq2$, let $\mathcal{O}\subset\pospathlag{n}$ be a Hamiltonian isotopy class, and let $\Lambda_0,\Lambda_1\in\mathcal{O}$. Then there is a unique geodesic connecting $\Lambda_0$ and $\Lambda_1$.
\end{thm}

Let $(X,\omega,J,\Omega)$ be an almost Calabi-Yau manifold and let $\Gamma \subset X$ be a Lagrangian submanifold. Denote by $\vol_\Gamma$ the volume form induced on $\Gamma$ by the K\"ahler metric associated to $\omega.$ It was shown in~\cite{harvey-lawson} that $\Omega|_\Gamma = \rho \vol_\Gamma$ for $\rho : \Gamma \to \mathbb{C}$ a non-vanishing smooth function. Following \cite{harvey-lawson}, the Lagrangian $\Gamma$ is called \emph{special} if $\rho$ has constant argument on $\Gamma.$  Let $\mathcal{O}$ be a Hamiltonian isotopy class of positive Lagrangians in $X.$ The main object of~\cite{solomon} is a functional $\mathcal{C}$ defined on the universal cover $\tilde{ \mathcal{O}}\to\mathcal{O}$, which is strictly convex along geodesics and which has a critical point over $\Gamma \in \mathcal{O}$ if and only if $\Gamma$ is special Lagrangian. Thus the following is an immediate consequence of Theorem~\ref{boundary}.
\begin{cor}
\label{uniquespecial}
Let $n\geq2$. Then in every Hamiltonian isotopy class $\mathcal{O}\subset\pospathlag{n}$, there is at most one special Lagrangian.
\end{cor}
Corollary~\ref{uniquespecial} was proved previously in the case $n=3$ in~\cite{shapere-vafa}. It is also a special case of a much more general theorem of Thomas-Yau~\cite{thomas-yau}, which relies on Floer cohomology.

A Riemannian metric is said to be \emph{strong} if it induces an isomorphism between the tangent and cotangent bundles, and \emph{weak} otherwise. As shown in \cite{klingenberg}, the distance function induced by a strong Riemannian metric makes the manifold a metric space. However, the Riemannian metric $\Upsilon$ is weak, and hence the nature of the distance function it induces is a priori unknown. In \cite{eliashberg-polterovich} and \cite{michor-mumford-vanish} there are examples of weak Riemannian structures with respect to which the distance between any two points is $0$. On the other hand, \cite{chen-kaehlermetrics} and \cite{clarke-metricgeometry} provide examples of weak Riemannian metrics which induce metric space structures. In this context, we prove the following.

\begin{thm}
	\label{metricspace}
	Let $n\geq2$, and let $\mathcal{O}\subset\pospathlag{n}$ be a Hamiltonian isotopy class. Then the Riemannian metric $\Upsilon$ induces a nondegenerate distance function $\overline{\Upsilon}$, making $\mathcal{O}$ into a metric space. Moreover, $(\mathcal{O},\overline{\Upsilon})$ is isometric to a convex subset of $C^\infty([0,1])$ with the $L^2$ metric.
\end{thm}
The fact that $(\mathcal{O},\overline{\Upsilon})$ is isometric to a convex subset of $L^2$ should be considered in the following context. The curvature formula of~\cite{solomon-curv} shows that a general Hamiltonian isotopy class of positive Lagrangians has non-positive curvature. However, $(\mathcal{O},\Upsilon)$ is atypical because of the $O_n(\mathbb{R})$ symmetry, and the curvature formula shows it is in fact flat. So, it follows from a Jacobi field argument that the exponential map of $(\mathcal{O},\Upsilon)$ is a Riemannian isometry. In practice, we prove Theorem~\ref{metricspace} more directly without using curvature.

As implied by Theorem~\ref{metricspace}, the metric space $(\mathcal{O},\overline{\Upsilon})$ is not complete. Thus, it would be of interest in future research to understand and characterize its completion. The completions of other weak Riemannian manifolds are explored in \cite{clarke-completion}, \cite{clarke-rubinstein-conformal} and \cite{clarke-rubinstein-ricci}.

In Sections~\ref{lagpaths}-\ref{laggeodesics} we review properties of Lagrangian paths and the definition of Lagrangian geodesics. In Section~\ref{milnorfiber} we give a more explicit treatment of $\pathlag{n}$ and characterize the elements which are positive Lagrangians. Section~\ref{lagpathsmilnor} discusses Lagrangian paths in $\pathlag{n}$. In Section~\ref{blowup} we review some basic properties of real blowups of surfaces. In Section~\ref{horfoliation} we blowup $M^1$ to obtain a new surface $\tilde{M}^1$, and construct a foliation on $\tilde{M}^1$ with isolated singular points, which is closely related to the Levi-Civita connection on $\pospathlag{n}$. In Sections~\ref{PHVF}-\ref{PHFG} we discuss pseudo-Hamiltonian flows and their relation to geodesics in $\pospathlag{n}$, and we prove Theorem~\ref{initial}. In Section~\ref{CF} we discuss cylindrical foliations and prove Theorem~\ref{boundary}. Theorem~\ref{metricspace} is proved in Section~\ref{metric}.

\section{Background}

\subsection{Lagrangian paths}
\label{lagpaths}

We start by summarizing the relevant definitions and arguments of \cite[Section~2]{akveld-salamon}, regarding paths of Lagrangian submanifolds, with some adjustment to our purposes.

Let $(X,\omega)$ be a symplectic manifold of dimension $2n$, $L$ a smooth manifold of dimension $n$, and $\mathcal{G}$ the diffeomorphism group of $L$. Denote by $\mathcal{X}$ the space of Lagrangian embeddings of $L$ in $X$. That is, $\mathcal{X}$ consists of embeddings $f:L\to X$ with $f^*\omega=0$. Note that $\mathcal{G}$ acts on $\mathcal{X}$ by
\begin{align*}
(\varphi,f)\mapsto f\circ\varphi.
\end{align*}
The space of Lagrangian submanifolds which are diffeomorphic to $L$ is then defined to be the quotient
\begin{align*}
\mathcal{L}=\mathcal{X}/\mathcal{G},
\end{align*}
where as implied by the name of this space, we think of every equivalence class in it as a Lagrangian submanifold. Given a path $\Lambda:(-\epsilon,\epsilon)\to\mathcal{L}$, a \emph{lift} of $\Lambda$ is a map
\begin{align*}
\Psi:(-\epsilon,\epsilon)\to\mathcal{X},
\end{align*}
or equivalently
\begin{align*}
\Psi:(-\epsilon,\epsilon)\times L\to X,
\end{align*}
such that for every $t$, the image of $\Psi(t)$ is $\Lambda(t)$. The path $\Lambda$ is then said to be \emph{smooth} if it has a smooth lift.

\begin{dfn}
Let $\Lambda$ be a smooth path as above, and let $\Psi$ be a smooth lift. \emph{The deformation vector field corresponding to $\Psi$ at time $t$} is the smooth section of the pullback bundle $\Psi_t^*TX$ given by
\begin{align*}
	v_t(p)=\pderiv[\Psi]{t}(t,p).
\end{align*}
\end{dfn}
The following lemma is proved in \cite[Section~2]{akveld-salamon}:
\begin{lemma}
\label{induced}
Let $\Lambda:(-\epsilon,\epsilon)\to\mathcal{L}$ be a smooth path, let $\Psi$ be some lift of $\Lambda$, and let $v_0$ denote the deformation vector field of $\Psi$ at time $t=0$. Define a $1$-form $\tilde{\sigma}$ on $L$ by
\begin{align*}
	\tilde{\sigma}_p(u)=\omega(v_0(p),d\Psi_{0,p}(u)),\qquad p\in L,u\in T_pL,
\end{align*}
and then define a $1$-form $\sigma$ on $\Lambda(0)$ by
\begin{align*}
	\sigma=\Psi_{0,*}\tilde{\sigma}.
\end{align*}
Then $\sigma$ is closed, and does not depend on the choice of the lift $\Psi$.  Furthermore, the map $\Lambda\mapsto\sigma$ yields an isomorphism of $T_{\Lambda(0)}\mathcal{L}$ and the space of closed $1$-forms on $\Lambda(0)$.
\end{lemma}
The \emph{derivative} of the path $\Lambda$ at $0$ is defined to be the closed $1$-form $\sigma$ of Lemma~\ref{induced}, and the derivative at any other time is defined similarly. $\Lambda$ is said to be \emph{exact} if its derivative at $t$ is an exact $1$-form on $\Lambda(t)$ for all $t$. As shown in \cite[Section~2]{akveld-salamon}, $\Lambda$ is exact if and only if there exists a Hamiltonian flow
\[
\phi_t:X\to X,\qquad t\in(-\epsilon,\epsilon),
\]
which satisfies \[\phi_t(\Lambda(0))=\Lambda(t)\] for every $t$.

\subsection{Geodesics of Lagrangian submanifolds}
\label{laggeodesics}

The following section is based on \cite[Section~5]{solomon} and \cite{solomon-curv}, and briefly recalls the Riemannian structure of a Hamiltonian isotopy class in the space of positive Lagrangians in an almost Calabi-Yau manifold.

\begin{dfn}
An \emph{almost Calabi-Yau} manifold is a quadruple $(X,\omega,J,\Omega)$, where $(X,\omega,J)$ is a K\"{a}hler manifold and $\Omega$ is a non-vanishing holomorphic form on $X$ of maximal rank.
\end{dfn}
\begin{dfn}
\label{poslag}
Let $(X,\omega,J,\Omega)$ be almost Calabi-Yau, and let $\Lambda\subset X$ be a Lagrangian submanifold. We say $\Lambda$ is \emph{positive} if $\real\Omega|_\Lambda$ is a volume form on $\Lambda$. Whenever integrating another volume form on a positive Lagrangian, we use the orientation determined by $\real\Omega$.
\end{dfn}

Let $(X,\omega,J,\Omega)$ be almost Calabi-Yau. Let $\mathcal{L}^+\subset\mathcal{L}$ denote the space of positive Lagrangian submanifolds of $X$ which are diffeomorphic to a compact connected manifold $L$. Let $\mathcal{O}\subset\mathcal{L}^+$ be a Hamiltonian isotopy class. Then every path in $\mathcal{O}$ is exact, and for every $\Lambda\in\mathcal{O}$ we have
\begin{align*}
T_\Lambda\mathcal{O}=\{d h|h\in C^\infty(\Lambda)\}.
\end{align*}
Since $\Lambda$ is a compact connected positive Lagrangian, the right hand space can be identified with
\begin{align*}
\mathcal{H}_\Lambda:=\left\{h\in C^\infty(\Lambda)\left|\int_\Lambda h\real\Omega=0\right\}\right..
\end{align*}
This enables us to define a Riemannian metric $\Upsilon$ on $\mathcal{O}$ by
\begin{align*}
\Upsilon(h,k)=\int_\Lambda hk\real\Omega,\quad h,k\in\mathcal{H}_\Lambda.
\end{align*}

It is shown in \cite{solomon-curv} that $\Upsilon$ has a Levi-Civita connection. The details are the following. Let $\Lambda:[0,1]\to\mathcal{O}$ be a smooth path, and let $(h_t),\;t\in[0,1]$, be a vector field along $\Lambda$. That is, $h_t\in\mathcal{H}_{\Lambda(t)}$ for every $t$. Pick a smooth lift $\Psi$, and let $v_t$ denote the deformation vector field of $\Psi$ at time $t$. Set
\[
\tilde{h}_t=h_t\circ\Psi_t.
\]
Write
\[
\tilde{\Omega}_t=\Psi_t^*\Omega,
\]
and let $w_t$ be the unique vector field on $L$ satisfying
\[
i_{w_t}\real\tilde{\Omega}_t=-i_{v_t}\real\Omega.
\]
Namely, for $p\in L$ and $u_1,\ldots,u_{n-1}\in T_pL$ we have
\[
\real\tilde{\Omega}_t(w_t(p),u_1,\ldots,u_{n-1})=-\real\Omega(v_t(p),d\Psi_t(u_1),\ldots,d\Psi_t(u_{n-1})).
\]
The covariant derivative is then given by
\begin{align}
\label{covariant}
\frac{D}{dt}h_t=\left(\deriv{t}\tilde{h}_t+d\tilde{h}_t(w_t)\right)\circ\Psi_t^{-1}.
\end{align}
It can be verified that the covariant derivative is independent of the lift $\Psi$ and hence well defined. 

The term \emph{geodesic of positive Lagrangians} is defined in the usual manner, with respect to the above connection. That is, a geodesic is a path $\Lambda:I\to\mathcal{O}$ satisfying
\[
\frac{D}{dt}\deriv{t}\Lambda(t)=0.
\]

We now recall the notion of horizontal lifts. As we shall see later, it can be useful for constructing geodesics.

\begin{dfn}
	Let $\Lambda:[0,1]\to\mathcal{O}$ be a path, let $\Psi$ be a smooth lift, and let $v_t$ denote the deformation vector field corresponding to $\Psi$ at time $t$. We say $v_t$ is \emph{\horizontal} if it satisfies
	\[
	i_{v_t}\real\Omega=0.
	\]
	The lift $\Psi$ is said to be \emph{horizontal} if $v_t$ is \horizontal\ for $t\in[0,1]$. As explained in~\cite{solomon},  a given embedding $\Psi_0:L\to X$ whose image is $\Lambda(0)$ extends uniquely to a horizontal lift of $\Lambda$. 
\end{dfn}
The geodesic equation becomes easier to handle using horizontal lifts. Indeed, let $\Lambda:[0,1]\to\mathcal{O}$, and let $\Psi$ be a horizontal lift. For $t\in[0,1]$, let $h_t\in\mathcal{H}_{\Lambda(t)}$ such that
\[
\deriv{t}\Lambda(t)=dh_t.
\]
Since $\Psi$ is horizontal, the vector field $w_t$ in \eqref{covariant} vanishes, and $\Lambda$ is a geodesic if and only if
\[
\deriv{t}h_t\circ\Psi_t(p)\equiv0,\qquad p\in L,
\]
or equivalently,
\[
h_t\circ\Psi_t=h_0\circ\Psi_0,\qquad t\in[0,1].
\]

\section{Lagrangian spheres in Milnor fibers}

\subsection{The Milnor fiber}
\label{milnorfiber}
Let $f$ be a complex polynomial in one variable whose zeros are all simple. For any positive integer $n$, the corresponding \emph{Milnor fiber} of dimension $n$ is given by
\begin{equation*}
M^n = \{(z_1,\ldots ,z_n,\zeta ) \in \mathbb{C}^{n+1}\  | \ z_1^2+\ldots +z_n^2=f(\zeta)\}.
\end{equation*}
As a complex submanifold of a K\"ahler manifold, $(M^n,\omega,J)$ is also K\"ahler, where $\omega$ and $J$ are the symplectic form and complex structure inherited from $\mathbb{C}^{n+1}$. Define
\begin{align}
\label{Omega}
\Omega=\frac{(-1)^{n+i}}{2z_i}\left(\bigwedge_{j\ne i}d z_j\right)\wedge d\zeta=\frac{1}{f'(\zeta)}\bigwedge_{j=1}^nd z_j.
\end{align}
Since all the zeros of $f$ are simple, $\Omega$ is a well defined nonvanishing holomorphic $(n,0)$ form on $M^n$, and $(M^n,\Omega)$ is almost Calabi-Yau.

We now recall matching cycles over curves. Given a complex number $\zeta$, set
\begin{displaymath}
\sigma_{\zeta}^{n-1}=\{(z_1,\ldots,z_n,\zeta)\in M^n\ |\ z_i\in \sqrt{f(\zeta)}\mathbb{R}\ \text{for }i=1,\ldots ,n\}.
\end{displaymath}
Note that if $\zeta$ is a zero of $f$, then $\sigma_{\zeta}^{n-1}$ consists of the single point $(0,\ldots,0,\zeta)$, and for any other $\zeta$, the locus $\sigma_{\zeta}^{n-1}$ is diffeomorphic to the $n-1$-sphere. Let $c\subset\mathbb{C}$ be a connected regular curve with boundary $\partial c=\{p,p'\}$, such that $f(p)=f(p')=0$, and $f(q)\neq0$ for every interior point $q\in c$. The matching cycle over $c$ is given by
\begin{displaymath}
\Lambda_c^n=\bigcup_{\zeta\in c}\sigma_{\zeta}^{n-1}.
\end{displaymath}
The following lemma is proved in \cite[Section~6]{khovanov-seidel}:
\begin{lemma}
\label{lagrangian}
For any regular curve $c\subset\mathbb{C}$ with the above requirements, $\Lambda_c^n$ is a smooth Lagrangian submanifold of $M^n$, diffeomorphic to the n sphere.
\end{lemma}
Note that for any curve $c$ as above, the matching cycle over $c$ is $O_n(\mathbb{R})$ invariant. On the other hand, it can be seen that any $O_n(\mathbb{R})$-invariant Lagrangian sphere in $M^n$ on which $O_n(\mathbb{R})$ does not act freely is a matching cycle over a curve. Thus $\pathlag{n}$ is an infinite dimensional manifold.

\begin{rem}
	\label{naive}
	The elements of $\pathlag{n}$ are actually called \emph{naive matching cycles}, according to Seidel's terminology in~\cite{seidel}. 
\end{rem}

\begin{rem}
	\label{s1}
	Throughout this paper, we think of $S^1$ as the quotient $\mathbb{R}/2\pi\mathbb{Z}$. Considering $S^1$ as a group, we let $-:S^1\to S^1$ denote the inversion. For example, the equalities $-(0)=0,-\pi=\pi$, hold in $S^1$ with our notation.
\end{rem}
\begin{dfn}
	A \emph{symmetric circle} is an embedding $\gamma:S^1\to M^1,\;u\mapsto(z(u),\zeta(u))$, which satisfies
	\begin{enumerate}[label=(\alph*)]
		\item\label{symmetrica} $\gamma$ is $O_1(\mathbb{R})$-equivariant. That is,
		\[
		\forall u\quad(z(-u),\zeta(-u))=(-z(u),\zeta(u)).
		\]
		\item\label{symmetricb} $\forall u\neq0,\pi\quad z(u)\neq0$.
	\end{enumerate}
	Note that a symmetric circle satisfies $z(0)=z(\pi)=0$, as implied by condition \ref{symmetrica}. Images of symmetric circles are just elements of $\pathlag{1}$.
\end{dfn}
The following fact will be used below.
\begin{lemma}
	\label{composesqrt}
	Let $g:(-\epsilon,\epsilon)\to\mathbb{R}$ be smooth and satisfy $g(-x)=g(x)$ for all $x$. Then the composition $h=g\circ\sqrt{\;}:[0,\epsilon^2)\to\mathbb{R}$ is infinitely differentiable from the right at $0$ and satisfies $h'(0)=g''(0)/2$.
\end{lemma}
\begin{lemma}
	\label{regular} \mbox{}
\begin{enumerate}[label=(\alph*)]
\item \label{regulara}Let $\alpha:(-\epsilon,\epsilon)\to M^1,\;x\mapsto(z(x),\zeta(x))$, be a smooth embedding satisfying
\begin{gather*}
z(0)=0,\quad z(x)\neq0\;\forall x\neq0,\\(z(-x),\zeta(-x))=(-z(x),\zeta(x))\;\forall x,
\end{gather*}
and let $c=\{\zeta(x)|x\in(-\epsilon,\epsilon)\}\subset\mathbb{C}$. Then the map $\eta:[0,\epsilon^2)\to c,\;x\mapsto\zeta(\sqrt{x}),$ is a diffeomorphism, and so $c$ is a regular curve with boundary $\partial c=\{\zeta(0)\}$.
\item \label{regularb}Let $\gamma:S^1\to M^1,\;u\mapsto(z(u),\zeta(u))$ be a symmetric circle. Then the set $c=\{\zeta(u)|u\in S^1\}\subset\mathbb{C}$ is a regular curve with boundary $\partial c=\{\zeta(0),\zeta(\pi)\}$.
\end{enumerate}
\end{lemma}
\begin{proof}\mbox{}
We start with \ref{regulara}. At any $(z,\zeta)\in M^1$ with $z\ne0$ the projection $(z,\zeta)\mapsto\zeta$ is a local diffeomorphism, and hence we only need to show regularity of $\eta$ at $0$. By assumption $\zeta$ is even, and thus $\eta$ is infinitely differentiable from the right at $0$, by Lemma~\ref{composesqrt}. We need to verify that $\eta'(0)\ne0$, or equivalently $\zeta''(0)\ne0$. Differentiating $(z(x))^2=f(\zeta(x))$ yields
\begin{align*}
2z(x)z'(x)=f'(\zeta(x))\zeta'(x),
\end{align*}
and differentiating again,
\begin{align*}
2(z'(x))^2+2z(x)z''(x)=f''(\zeta(x))(\zeta'(x))^2+f'(\zeta(x))\zeta''(x).
\end{align*}
Since $z(0)=\zeta'(0)=0$, substituting $x=0$ yields
\begin{align*}
2(z'(0))^2=f'(\zeta(0))\zeta''(0).
\end{align*}
Since $\alpha$ is an embedding and $\zeta'(0)=0$, we conclude $z'(0)\ne0$, and by the last equation $\eta$ is indeed regular.

For \ref{regularb}, once again, we only need to show regularity of $c$ near $\zeta(0),\zeta(\pi)$. Identifying a neighborhood of $0\in S^1$ with a real interval $(-\epsilon,\epsilon)$, regularity at $\zeta(0)$ follows from~\ref{regulara}. Regularity at $\zeta(\pi)$ is shown similarly.\qedhere
\end{proof}

Lemmas~\ref{lagrangian} and~\ref{regular} \ref{regularb} allow us to forget about curves in the complex plane, and consider the corresponding elements in $\pathlag{1}$ instead. The letter $\gamma$ will be used henceforth to denote a symmetric circle, and the two coordinates of $\gamma$ will be denoted by $z,\zeta$. By abuse of notation, we shall refer to the corresponding matching cycle in $M^n$ as $\Lambda_\gamma^n$, instead of $\Lambda_c^n$.
\begin{prop}
	\label{spherediffeo}
	Let $\gamma$ be a symmetric circle. Set
	\begin{align*}
	\varphi:S^1\times S^{n-1}\to\Lambda_\gamma^n,\quad(u,x)\mapsto(z(u)\cdot x,\zeta(u))
	\end{align*}
	and
	\begin{align*}
	\chi:S^1\times S^{n-1}\to S^n,\quad(u,x)\mapsto(\sin(u)\cdot x,\cos(u)).
	\end{align*}
	Then there is a unique map $\Psi:S^n\to\Lambda_\gamma^n$ which satisfies $\Psi\circ\chi=\varphi$, and it is a diffeomorphism.
\end{prop}
\begin{proof}
	Existence and uniqueness of $\Psi$ follow from the universal property of quotient, and so does bijectivity. Both $\varphi$ and $\chi$ are local diffeomorphisms at any $(u,x)$ with $u\neq0,\pi$. Hence $\Psi$ is smooth and regular anywhere away from the two poles. Let $D^n$ stand for the open $n$-dimensional unit ball, and parametrize the north $n$-hemisphere by
	\begin{align*}
	\mathbf{Y}:D^n\to S^n,\quad y\mapsto(y,\sqrt{1-|y|^2}).
	\end{align*}
	We show that $\Psi\circ\mathbf{Y}$ is smooth and regular at $y=0$. By assumption on $\gamma$, there is a smooth even nonvanishing function $r:S^1\to\mathbb{C}$ such that
	\begin{align*}
	\forall u\in S^1\quad z(u)=r(u)\sin(u).
	\end{align*}
	Let $U\subset S^1$ be a small neighborhood of $0$. If $y\in D^n$ and $(u,x)\in U\times S^{n-1}$ satisfy
	\begin{align*}
	\mathbf{Y}(y)=\chi(u,x)
	\end{align*}
	then
	\begin{align*}
	u=\pm\arcsin|y|,
	\end{align*}
	and it follows that
	\begin{align*}
	\Psi\circ\mathbf{Y}(y)=(r(\arcsin|y|)\cdot y,\zeta(\arcsin|y|)).
	\end{align*}
	The function $y\mapsto r(\arcsin|y|)$ can be written as $y\mapsto r\circ\arcsin\circ\sqrt{\;}(|y|^2)$, and it is smooth by Lemma~\ref{composesqrt}, as $r\circ\arcsin$ is smooth and even. The function $y\mapsto\zeta(\arcsin|y|)$ is smooth by a similar argument. Hence $\Psi\circ\mathbf{Y}$ is smooth. It is regular at $y=0$ since $r$ is nonvanishing. Smoothness and regularity of $\Psi$ at the other pole can be shown similarly.
\end{proof}

\begin{rem}
	\label{realhol}
	All the tangent vectors mentioned in this article are to be understood as real tangent vectors. However, keeping in mind the canonical identification of the real and holomorphic tangent bundles of a complex manifold, it is often convenient to use the notation of holomorphic tangent vectors. Thus, given a holomorphic coordinate $z=x+iy$ we write $\partial_z$ for $\partial_x$ and $i\partial_z$ for $\partial_y$.
\end{rem}

\begin{rem}
	\label{frame}
	Given a symmetric circle $\gamma$, the following is a convenient way to work with the tangent space to $\Lambda_\gamma^n$ at any point. One can embed the circle $\Lambda_\gamma^1$ in $\Lambda_\gamma^n$ by
	\begin{align*}
		i:\Lambda_\gamma^1  \hookrightarrow\Lambda_\gamma^n,\quad
		(z,\zeta)  \mapsto(z,0,\ldots,0,\zeta).
	\end{align*}
	Define $n$ vector fields along $i$, $X_1,\ldots,X_n$, as follows. For any $u\in S^1$ set
	\begin{align*}
		X_1(\gamma(u))=d i_{\gamma(u)}(\gamma'(u))=z'(u)\partial_{z_1}+\zeta'(u)\partial_\zeta.
	\end{align*}
	For $j=2,\ldots,n$, set
	\begin{align*}
		X_j(\gamma(u))=\left\{
		\begin{array}{rl}
			\frac{z(u)}{\sin(u)}\partial_{z_j} & u\ne0,\pi \\
			z'(0)\partial_{z_j} & u=0 \\
			-z'(\pi)\partial_{z_j} & u=\pi.
		\end{array}\right.
	\end{align*}
	It is easy to verify that for $j=1,\ldots,n$, the vector field $X_j$ is tangent to $\Lambda_\gamma^n$ and smooth. Since $X_1,\ldots,X_n$, are linearly independent, they form a smooth frame of the pullback bundle $i^*T\Lambda_\gamma^n$. Furthermore, as verified by a straightforward calculation, for $u\ne0,\pi$, the tangent vectors $X_j(\gamma(u))$, $j=2,\ldots,n$, form a basis of $T_{i(\gamma(u))}\sigma_{\zeta(u)}^{n-1}$. Note that $\sigma_{\zeta(u)}^{n-1}$ is the $O_n(\mathbb{R})$ orbit of $i(\gamma(u))$, and the above construction can be used to obtain a basis of $T_p\Lambda_\gamma^n$ for any $p\in\Lambda_\gamma^n$.
\end{rem}

The following is a basic observation which will be used later, in the proof of Theorem~\ref{boundary}.
\begin{lemma}
	\label{smoothonsphere}
	Let $\gamma$ be a symmetric circle, and let $i:\Lambda_\gamma^1\hookrightarrow\Lambda_\gamma^n$ as in Remark~\ref{frame}. Let $g:\Lambda_\gamma^n\to\mathbb{R}$ be invariant under the $O_n(\mathbb{R})$ action, such that $g\circ i:\Lambda_\gamma^1\to\mathbb{R}$ is smooth. Then $g$ is smooth.
\end{lemma}
\begin{proof}
	Let $c=\{\zeta(u)|u\in S^1\}\subset\mathbb{C}$, and let $P:\Lambda_\gamma^1\to c$ be the projection. Since $g\circ i$ is invariant under the $O_1(\mathbb{R})$ action, there is some $h:c\to\mathbb{R}$ such that $g\circ i=h\circ P$. We claim that $h$ is smooth. Since $P$ has a smooth local inverse anywhere away from the boundary of $c$, we only need to show smoothness of $h$ at the endpoints. Let $V\subset c$ be a small open neighborhood of $\zeta(0)$, and identify a small neighborhood of $0\in S^1$ with a real interval $(-\epsilon,\epsilon)$. By Lemma~\ref{regular} \ref{regulara}, the map $u\mapsto\zeta(\sqrt{u})$ defined on $[0,\epsilon)$ is a diffeomorphism, and hence the function $\zeta(u)\mapsto u^2$ maps $V$ smoothly into $S^1$. By Lemma~\ref{composesqrt}, $g\circ i$ depends smoothly on $u^2$, and $h$ is thus smooth at $\zeta(0)$. A similar argument shows smoothness at $\zeta(\pi)$. The projection $\Lambda_\gamma^n\to c$ is also smooth. The lemma follows.
\end{proof}

As our main objective is Lagrangian geodesics, we need to characterize the circles $\gamma$ whose corresponding Lagrangians $\Lambda_\gamma^n$ lie in $\pospathlag{n}$, the space of \emph{positive} Lagrangians in $\pathlag{n}$.
\begin{lemma}
\label{positive}
Let $\gamma$ be a symmetric circle. Then $\Lambda_\gamma^n$ is a positive Lagrangian if and only if the two following conditions hold:
\begin{enumerate}[label=(\alph*)]
\item\label{positivea} $\quad z'(u)(z(u))^{n-1}/f'(\zeta(u))=\zeta'(u)(z(u))^{n-2}/2\not\in i\mathbb{R}$ for all $u\ne0,\pi$.
\item\label{positiveb} $\quad(z'(u))^n/f'(\zeta(u))\not\in i\mathbb{R}$ for $u=0,\pi$.
\end{enumerate}
\end{lemma}
\begin{proof}
Using the frame given in Remark~\ref{frame}, a direct calculation shows that for $u\ne0,\pi,$
\begin{align*}
\Omega(X_1(\gamma(u)),\ldots,X_n(\gamma(u)))=&\frac{z'(u)}{f'(\zeta(u))}\left(\frac{z(u)}{\sin(u)}\right)^{n-1}\\ & =\frac{\zeta'(u)}{2z(u)}\left(\frac{z(u)}{\sin(u)}\right)^{n-1} \\ & = \frac{\zeta'(u)(z(u))^{n-2}}{2\sin^{n-1}(u)}\;,
\end{align*}
and for $u=0,\pi$,
\begin{align*}
\Omega(X_1(\gamma(u)),\ldots,X_n(\gamma(u)))=\pm\frac{(z'(u))^n}{f'(\zeta(u))}\;.
\end{align*}
It follows that conditions \ref{positivea} and \ref{positiveb} are equivalent to $\real\Omega$ not vanishing on $i(\Lambda_\gamma^1)$. Since any point in $\Lambda_\gamma^n$ lies in the $SO_n(\mathbb{R})$ orbit of some $p\in i(\Lambda_\gamma^1)$ and $\Omega$ is $SO_n(\mathbb{R})$ invariant, the conditions are equivalent to $\real\Omega$ not vanishing anywhere on $\Lambda_\gamma^n$.
\end{proof}

\subsection{Lagrangian paths and horizontal lifts}
\label{lagpathsmilnor}

As any symmetric circle $\gamma$ has a  corresponding Lagrangian $n$-sphere $\Lambda_\gamma^n$, a smooth family of symmetric circles $(\gamma_t)$, with $t$ varying along a real interval, gives rise to the Lagrangian path $(\Lambda_{\gamma_t}^n)$. The following is an important property of such paths.

\begin{prop}
\label{paths}
For every $\Lambda\in\pathlag{n}$,
\begin{align*}
T_\Lambda\pathlag{n}\subset\{d h|h:\Lambda\to\mathbb{R}\;\mathrm{is}\;O_n(\mathbb{R})\;\mathrm{invariant}\}.
\end{align*}
In particular, every path in $\pathlag{n}$ is exact.
\end{prop}
\begin{rem}
\label{promisedeq}
It will be shown in the proof of Theorem~\ref{initial} that for $\Lambda\in\pospathlag{n}$ we have the equality
\begin{align*}
T_\Lambda\pathlag{n}=\{d h|h:\Lambda\to\mathbb{R}\;\mathrm{is}\;O_n(\mathbb{R})\;\mathrm{invariant}\}.
\end{align*}
\end{rem}
\begin{proof}[Proof of Proposition~\ref{paths}]
	Let $(\gamma_t:S^1\to M^1),\;t\in(-\epsilon,\epsilon),$ be
    a smooth family of symmetric circles with $\Lambda_{\gamma_0}^n=\Lambda$, and for convenience let
    \begin{align*}
    \sigma:=\left.\deriv{t}\right|_{t=0}\Lambda_{\gamma_t}^n.
    \end{align*}
    Note that by definition of the derivative, the $1$-form $\sigma$ is $O_n(\mathbb{R})$ invariant, as the symplectic form $\omega$ is, and so is $\Lambda_{\gamma_t}^n$ for every $t$. We claim that $\sigma$ is exact. For $n\geq2$ it is immediate since $\sigma$ is closed and $\Lambda$ is simply connected. If $n=1$, invariance under the nontrivial element in $O_1(\mathbb{R})$ yields
    \begin{align*}
    \int_\Lambda \sigma=0,
    \end{align*}
    and exactness follows.

    We show $O_n(\mathbb{R})$ invariance of the primitive. Let $\tau$ be a volume form on $\Lambda$ which satisfies
    \begin{align*}
    	\forall A\in O_n(\mathbb{R})\quad A^*\tau=\det(A)\cdot\tau,
    \end{align*}
    and let $h:\Lambda\to\mathbb{R}$ satisfy $d h=\sigma$. Since $\sigma$ is $O_n(\mathbb{R})$ invariant, for $A\in O_n(\mathbb{R})$ we have
    \begin{align*}
    d(h\circ A-h)=d(A^*h-h)=A^*\sigma-\sigma=0.
    \end{align*}
    Since $\Lambda$ is connected, it follows from the last equation that the two functions $h\circ A$ and $h$ differ by a constant. But
    \begin{align*}
    \int_\Lambda h\circ A\wedge\tau=\det(A)\int_\Lambda A^*(h\wedge\tau)=\int_\Lambda h\wedge\tau,
    \end{align*}
    which yields $h\circ A=h$.
\end{proof}
Our next goal is to understand \emph{horizontal lifts} of paths in $\pospathlag{n}$.
\begin{rem}
	\label{fields}
	Let $O_n(\mathbb{R})$ act on $S^n$ by
	\begin{align*}
		\left(A,\left(
		\begin{array}{c}
			y_1\\\vdots\\y_{n+1}
		\end{array}\right)\right)\mapsto
		\left(
		\begin{array}{c}
			A\cdot\left(
			\begin{array}{c}
				y_1\\\vdots\\y_n
			\end{array}\right)\\
			y_{n+1}
		\end{array}
		\right).
	\end{align*}
	Let $(\gamma_t),\;t\in(-\epsilon,\epsilon)$, be a family of symmetric circles, and for every $t$ let $\Psi_t:S^n\to\Lambda_{\gamma_t}^n$ be the diffeomorphism induced by $\gamma_t$ as in Proposition~\ref{spherediffeo}. Note that $\Psi_t$ is $O_n(\mathbb{R})$ equivariant. The map $\Psi:(-\epsilon,\epsilon)\times S^n\to M^n$ is a lift of the Lagrangian path $(\Lambda_{\gamma_t}^n)$. Let $\xi_t$ denote the deformation vector field corresponding to $\Psi$ at time $t$. Note that for $t\in(-\epsilon,\epsilon)$ the pullback bundle $\Psi_t^*TM^n$ is $O_n(\mathbb{R})$ equivariant, and $\xi_t$ is $O_n(\mathbb{R})$ invariant.
\end{rem}

\begin{rem}
	\label{capitali}
	When analyzing paths in $\pathlag{n}$, it is rather convenient to have a standard embedding of $M^1$ in $M^n$. The letter $I$ will be used below for the embedding
	\begin{align*}
		I:M^1\hookrightarrow M^n,\quad(z,\zeta)\mapsto(z,0,\ldots,0,\zeta).
	\end{align*}
	Note that $I$ is an extension of the embedding $i$ of Remark~\ref{frame}.
\end{rem}

\begin{prop}
	\label{horizontal}
	Let $(\gamma_t),\;t\in(-\epsilon,\epsilon)$, be as above, such that $\Lambda_{\gamma_t}^n\in\pospathlag{n}$ for all $t$, and let $\Psi:(-\epsilon,\epsilon)\times S^n\to M^n$ be the lift of $(\Lambda_{\gamma_t}^n)$ induced by $(\gamma_t)$ as in Remark~\ref{fields}. Then $\Psi$ is horizontal if and only if
	\begin{align*}
	\pderiv[z]{t}\cdot\frac{(z_t(u))^{n-1}}{f'(\zeta_t(u))}=\pderiv[\zeta]{t}\cdot\frac{(z_t(u))^{n-2}}{2}\in i\mathbb{R}
	\end{align*}
	for all $u,t$.
\end{prop}
\begin{proof}
We embed $S^1$ in $S^n$ by
\begin{align*}
	e:S^1\hookrightarrow S^n,\quad u\mapsto(\sin(u),0,\ldots,0,\cos(u)).
\end{align*}

Let $\xi_t$ denote the deformation vector field corresponding to $\Psi$ at time $t$. It follows from Remark~\ref{fields} and $SO_n(\mathbb{R})$ invariance of $\Omega$ that $\xi_t$ is \horizontal if and only if it is \horizontal on $e(S^1)$. For $u_0\in S^1$, the space of $(n-1)$-vectors tangent to $\Lambda_{\gamma_t}^n$ at $I(\gamma_t(u_0))$ is spanned by $X_1(\gamma_t(u_0))\wedge\ldots\wedge\widehat{X_j(\gamma_t(u_0))}\wedge\ldots\wedge X_n(\gamma_t(u)_0),\;j=1,\ldots,n,$ where the hat indicates an omitted argument. Hence, $\xi_t$ is horizontal at $e(u_0)$ if and only if we have
\begin{equation}
	\label{horeq}
	\Omega(\xi_t(e(u_0)),X_1(\gamma_t(u_0)),\ldots,\widehat{X_j(\gamma_t(u_0))},\ldots,X_n(\gamma_t(u_0)))\in i\mathbb{R},\quad j=1,\ldots,n.
\end{equation}
Since $\xi_t(e(u_0))$ and $X_1(\gamma_t(u_0))$ are both tangent to the $1$-dimensional complex submanifold $I(M^1)\subset M^n$, these two vectors are collinear over $\mathbb{C}$. As the form $\Omega$ is of type $(n,0)$, this means that the expression in~\eqref{horeq} vanishes for $j=2,\ldots,n$, with no condition on the family $(\gamma_t)$. Thus, $\xi_t$ is horizontal if and only if~\eqref{horeq} holds for $j=1$. Substituting
\[
\Omega=\pm\frac{1}{2z_1}dz_2\wedge\ldots\wedge dz_n\wedge d\zeta=\frac{1}{f'(\zeta)}dz_1\wedge\ldots\wedge dz_n
\]
and
\[
X_k(\gamma_t(u_0))=\frac{z(u_0)}{\sin u_0}\partial_{z_k},\quad k=2,\ldots,n,
\]
this is equivalent to the desired condition.
\end{proof}

\subsection{Real blowup}
\label{blowup}
For our next constructions we shall use real blowups, which are described shortly below. We discuss only blowups of manifolds of dimension $2,$ as this is the only relevant case for our purposes.

Let $D$ denote the open unit ball in $\mathbb{R}^2$. The \emph{real blowup} of $D$ is defined by
\begin{align*}
\tilde{D}=\{(p,l)\in D\times\mathbb{R}\mathbb{P}^1|p\in l\},
\end{align*}
with the projection
\begin{align*}
\pi:\tilde{D}  \to D,\quad
(p,l)  \mapsto p.
\end{align*}
The \emph{exceptional divisor}, $E=\pi^{-1}(0)$, is just a copy of the real projective line, whereas the restriction of $\pi$ to the complement of $E$,
\begin{align*}
\pi|_{\tilde{D}\setminus E}:\tilde{D}\setminus E\to D\setminus0,
\end{align*}
is a diffeomorphism.

We shall use polar coordinates to parametrize $\tilde{D}$ in the following manner. For any $0<\delta<1$ and $a\in S^1$, define
\begin{equation*}
\mathbf{X}_{\delta,a}:(-\delta,\delta)\times(a,a+\pi)  \to\tilde{D},
\end{equation*}
by
\begin{equation*}
(r,\theta)  \mapsto((r\cos\theta,r\sin\theta),[\cos\theta:\sin\theta]).
\end{equation*}
Thus, $X_{\delta,a}$ is a parametrization of
\begin{align*}
U_{\delta,a}:=\{(p,l)\in\tilde{D}|l\ne[\cos a:\sin a],|p|<\delta\}.
\end{align*}
This is similar to the standard use of polar coordinates, but note that here we allow $r$ to admit negative values. The composition of $\pi$ with $\mathbf{X}_{\delta,a}$ is given by $(r,\theta)\mapsto(r\cos\theta,r\sin\theta)$. The intersection of the exceptional divisor and $U_{\delta,a}$ is given by
\begin{align*}
E\cap U_{\delta,a}=\{r=0\}=\{(0,l)|l\ne[\cos a:\sin a]\}.
\end{align*}
Any $\mathbf{X}_{\delta,a}$ as above will be called below \emph{a polar parametrization}, and $U_{\delta,a}$ will be called \emph{a polar open subset}. Note that $\tilde{D}$ can be covered by two open polar subsets.

Any smooth vector field on $D\setminus\{0\}$ can be pulled back to $\tilde{D}\setminus E$ by $(d\pi)^{-1}$. We express this pullback in coordinates. The tangent space to $\tilde{D}$ at any point in a polar open subset is spanned by the vector fields $\partial_r,\partial_\theta$, that are induced by a corresponding polar parametrization. Differentiating $\pi$ yields
\begin{align*}
d\pi(\partial_r) & =\cos\theta\partial_x+\sin\theta\partial_y, \\
d\pi(\partial_\theta) & =-r\sin\theta\partial_x+r\cos\theta\partial_y,
\end{align*}
and by inverting the corresponding matrix anywhere outside the exceptional divisor we get
\begin{align}
\label{pulltoblowup}
d\pi^{-1}(\partial_x) & =\cos\theta\partial_r-\frac{\sin\theta}{r}\partial_\theta, \nonumber \\
d\pi^{-1}(\partial_y) & =\sin\theta\partial_r+\frac{\cos\theta}{r}\partial_\theta\;.
\end{align}

Recall the strict transform of a curve. Given a smooth curve $c$ in $D$, that passes through the origin, \emph{the strict transform of $c$} is given by
\begin{align*}
\tilde{c}=\overline{(\pi|_{\tilde{D}\setminus E})^{-1}(c)}=(\pi|_{\tilde{D}\setminus E})^{-1}(c)\cup\{(0,T_0c)\}.
\end{align*}
If $c$ does not pass through the origin, the strict transform $\tilde{c}$ is nothing more than the inverse image of $c$ under $\pi$. We will use the following well known facts later:
\begin{lemma}\mbox{}
\label{liftcurve}
\begin{enumerate}[label=(\alph*)]
\item\label{liftcurvea}   Let $c$ be a smooth curve in $D$, and $\tilde{c}$ the strict transform. Then $\tilde{c}$ is a smooth curve in $\tilde{D}$, which intersects the exceptional divisor transversally.
\item\label{liftcurveb}  Let $\tilde{c}$ be a smooth curve in $\tilde{D}$ that intersects the exceptional divisor transversally, and set $c=\pi(\tilde{c})$. If the intersection $\tilde{c}\cap E$ consists of at most one point, $c$ is a smooth curve in $D$.
\end{enumerate}
\end{lemma}

We end this section with blowups of general surfaces. Let $M$ be a smooth $2$-dimensional manifold and $\{p_j\}\subset M$ a discrete subset. The blowup manifold $\tilde{M}$ of $M$ at $\{p_j\}$ is obtained by replacing a small disk around each $p_j$ by its real blowup. Gluing together all projections of the involved blown up disks with the identity map of $M$ yields the blowup projection
\begin{align*}
	\pi:\tilde{M}\to M.
\end{align*}

\subsection{The horizontal foliation}
\label{horfoliation}

Motivated by Proposition~\ref{horizontal}, we define a subbundle
\begin{align*}
\tau\subset T(M^1\setminus\{z=0\})
\end{align*}
by
\begin{align}
\label{subbundle}
\tau(p)=\{v\in T_pM^1|d\zeta(v)\in iz^{2-n}\mathbb{R}\},
\end{align}
and denote by $F$ the foliation tangent to $\tau$. Recall that a family $(\gamma_t)$ induces a lift of the Lagrangian path $(\Lambda_{\gamma_t}^n)$ as in Remark~\ref{fields}. By Proposition~\ref{horizontal}, the induced lift is horizontal if and only if for any fixed $u\in S^1,\;u\ne0,\pi,$ the path $(\gamma_t(u))$ is contained in a leaf of $F$. This relates $F$ closely to horizontal lifts of paths of Lagrangian spheres, and eventually to the construction of geodesics. The problem, though, is that all the symmetric circles $\gamma$ in discussion pass through points with $z=0$, which are singularities of $F$. As it turns out, this problem is solved by using the real blowup.

Let $q_1=(0,\zeta_1),\ldots,q_k=(0,\zeta_k)$, denote all points in $M^1$ with $z=0$, and let $\tilde{M}^1$ denote the manifold obtained by blowing $M^1$ up at $q_1,\ldots,q_k$, with the projection $\pi:\tilde{M}^1\to M^1$. Let $E_1,\ldots,E_k\subset\tilde{M}^1$ denote the respective exceptional divisors. For $j=1,\ldots,k,$ set
\begin{align}
	\label{zeros}
	Z_j=\left\{(0,\theta)\left|n\theta-\arg(f'(\zeta_j))\in\frac{\pi}{2}+\pi\mathbb{Z}\right.\right\}\subset E_j,
\end{align}
using polar coordinates over $q_j$ compatible with the $z$ coordinate (that is, $z=re^{i\theta}$). Note that compatibility with $z$ determines $\theta$ up to $\pi$, and $Z_j$ is thus well defined. Set
\begin{align*}
	H:=\bigcup_{j=1}^kZ_j.
\end{align*}

\begin{prop}
\label{foliation}
There exists a unique foliation $\tilde{F}$ of $\tilde{M}^1\setminus H$, such that $\pi$ carries the leaves of $\tilde{F}$ to the leaves of $F$. The foliation $\tilde{F}$ is tangent to all exceptional divisors $E_1,\ldots,E_k$. Furthermore, it satisfies the following properties:
\begin{enumerate}[label=(\alph*)]
\item\label{foliationa} Given a symmetric circle $\gamma$, the Lagrangian $\Lambda_\gamma^n$ is positive if and only if the strict transform $\tilde{\gamma}$ does not pass through any point in $H$ and intersects all leaves of $\tilde{F}$ transversally.
\item\label{foliationb} Let $(\gamma_t)$ be a family of symmetric circles such that for every $t$ the Lagrangian sphere $\Lambda_{\gamma_t}^n$ is positive. Then the lift induced by $(\gamma_t)$ is horizontal if and only if for every fixed $u\in S^1$ the curve $t\mapsto\tilde{\gamma}_t(u)$ is contained in a leaf of $\tilde{F}$, where $\tilde{\gamma}_t$ is the strict transform of $\gamma_t$.
\end{enumerate}
\end{prop}
\begin{proof}
Since $\pi|_{\tilde{M}^1\setminus\bigcup E_j}$ is a diffeomorphism, we define the leaves of $\tilde{F}$ anywhere away from $E_1,\ldots,E_k$ to be the strict transforms of the leaves of $F$. Uniqueness of $\tilde{F}$ follows from $\tilde{M}^1\setminus\cup E_j$ being dense in $\tilde{M}^1$. In order to show that $\tilde{F}$ can be extended to $\tilde{M}^1\setminus H$, we observe $F$ close to $q_1,\ldots,q_k$.

For any $\epsilon>0$, let $D_\epsilon$ denote the open ball of radius $\epsilon$ in the complex plane, centered at $0$. For $j=1,\ldots,k$, let $V_j$ be an open neighborhood of $q_j$, in which $\zeta$ is a smooth function of $z$, and let $\epsilon_j$ be small enough so that $D_{\epsilon_j}$ is contained in the image of $V_j$ under the projection $(z,\zeta)\mapsto z$. Thus for every $j$ there is an embedding
\begin{align*}
\varphi_j:D_{\epsilon_j}  \hookrightarrow V_j\subset M^1,\quad
z  \mapsto(z,\zeta(z)),
\end{align*}
and without loss of generality we assume $V_j=\varphi_j(D_{\epsilon_j})$. For convenience, we would like to omit the $\partial_\zeta$ component when describing a vector field on $V_j$, so we denote
\begin{align*}
	\hat{\partial}_z:=d\varphi_j(\partial_z)=\partial_z+\frac{2z}{f'(\zeta)}\partial_\zeta.
\end{align*}

Let $U\subset\pi^{-1}(V_j)$ be a polar open subset with polar coordinates $r,\theta$, compatible with the $z$ coordinate. On $V_j\setminus\{q_j\}$ the subbundle $\tau$ may be expressed by
\begin{align*}
\tau(p)=\{v\in T_pM^1|d z(v)\in if'(\zeta)z^{1-n}\mathbb{R}\}.
\end{align*}
Hence inside $\pi(U)\setminus\{q_j\}$ the leaves of $F$ can be thought of as the trajectories of the nonvanishing vector field
\begin{align}
\label{field1}
X(z)=if'(\zeta)z^{1-n}r^n\hat{\partial}_z=if'(\zeta)re^{i(1-n)\theta}\hat{\partial}_z.
\end{align}
Since $f'(\zeta)$ does not vanish in $V_j$, the functions $|f'(\zeta)|$ and $\arg(f'(\zeta))$ are well defined and smooth in $z$. We write
\begin{align*}
\acute{r} & :=|f'(\zeta)|\cdot r, \\
\acute{\theta} & :=\frac{\pi}{2}+\arg(f'(\zeta))+(1-n)\theta.
\end{align*}
Let $x,y,$ be the real coordinates on $V_j$ corresponding to $z$. Moving from complex to real notation (see Remark~\ref{realhol}), equation~\eqref{field1} can be rewritten as
\begin{align*}
X(z)=\acute{r}e^{i\acute{\theta}}\hat{\partial}_z=\acute{r}\cos\acute{\theta}\partial_x+\acute{r}\sin\acute{\theta}\partial_y.
\end{align*}
Pulling $X$ back to $U\setminus E_j$, we obtain from \eqref{pulltoblowup}
\begin{align*}
d\pi^{-1}(X) & =\acute{r}\cos\acute{\theta}\cos\theta\partial_r-\acute{r}\cos\acute{\theta}\frac{\sin\theta}{r}\partial_\theta+\acute{r}\sin\acute{\theta}\sin\theta\partial_r+\acute{r}\sin\acute{\theta}\frac{\cos\theta}{r}\partial_\theta \\
& =\acute{r}\cos(\acute{\theta}-\theta)\partial_r+\frac{\acute{r}}{r}\sin(\acute{\theta}-\theta)\partial_\theta \\
& =|f'(\zeta)|\cdot r\cos\left(\frac{\pi}{2}+\arg(f'(\zeta))-n\theta\right)\partial_r\\
& +|f'(\zeta)|\sin\left(\frac{\pi}{2}+\arg(f'(\zeta))-n\theta\right)\partial_\theta.
\end{align*}
It follows that $d\pi^{-1}(X)$ can be extended smoothly to a vector field $\tilde{X}$ which is defined on the whole of $U$. Note that along $E_j$ the coefficient of $\partial_r$ in $\tilde{X}$ vanishes. In other words, $\tilde{X}$ is tangent to $E_j$. The set of points where $\tilde{X}$ vanishes is exactly $Z_j$, as defined in \eqref{zeros}. We define the desired $\tilde{F}$ along $U\setminus Z_j$ to be the foliation tangent to $\tilde{X}$. Note that up to sign, the vector field $\tilde{X}$ does not depend on the choice of the polar open subset $U$. Hence $\tilde{F}$ is well defined along $\tilde{M^1}\setminus H$. It is tangent to every exceptional divisor $E_j$ since the corresponding vector field $\tilde{X}$ is.

We show that $\tilde{F}$ possesses the required properties, beginning with \ref{foliationa}.
Let $\gamma,\;u\mapsto(z(u),\zeta(u)),$ be such that the Lagrangian $\Lambda_\gamma^n$ is positive. By condition \ref{positiveb} of Lemma~\ref{positive}, $(z'(0))^n/f'(\zeta(0))$ and $(z'(\pi))^n/f'(\zeta(\pi))$ are not pure imaginary. Hence, by \eqref{zeros} we have $\tilde{\gamma}(0),\tilde{\gamma}(\pi)\not\in H$, and $\tilde{\gamma}$ does not pass through any singular point of $\tilde{F}$. By condition \ref{positivea} of Lemma~\ref{positive}, for any $u\ne0,\pi$, the product $\zeta'(u)\cdot z(u)^{n-2}$ is not pure imaginary. Hence $\gamma'(u)\not\in\tau(\gamma(u))$, and $\gamma$ is not tangent to any leaf of $F$. Since $\tilde{\gamma}$ is the strict transform of $\gamma$, it intersects the exceptional divisors over $\gamma(0)$ and $\gamma(\pi)$ transversally by Lemma~\ref{liftcurve}, and so it is transverse to all leaves of $\tilde{F}$. Conversely, assume that the strict transform $\tilde{\gamma}$ does not pass through any singular point of $\tilde{F}$ and is not tangent to any of its leaves. Then the two conditions of Lemma~\ref{positive} hold, and $\Lambda_\gamma^n$ is positive.

We verify property \ref{foliationb}. Let $(\gamma_t)$ be such that $\Lambda_{\gamma_t}^n\in\pospathlag{n}$ for all $t$. By construction, for a fixed $u\in S^1$, the path $t\mapsto\tilde{\gamma}_t(u)$ is contained in a leaf of $\tilde{F}$ if and only if for every $u\ne0,\pi$, the path $t\mapsto\gamma_t(u)$ is contained in a leaf of $F$. By Proposition~\ref{horizontal}, the latter is equivalent to $(\gamma_t)$ inducing a horizontal lift.\qedhere
\end{proof}

The following definition makes sense in view of Proposition~\ref{foliation}.

\begin{dfn}
	Let $(\gamma_t)$ be a family of symmetric circles, such that for every $t$ the Lagrangian sphere $\Lambda_{\gamma_t}^n$ is positive. We say the family $(\gamma_t)$ is \emph{horizontal} if for a fixed $u\in S^1$, the path $t\mapsto\tilde{\gamma}_t(u)$ is contained in a leaf of $\tilde{F}$.
\end{dfn}

\begin{lemma}
	\label{horizontalcircles}
	Let $\Lambda:[0,1]\to\pospathlag{n}$ be a smooth path, and let $\gamma_0$ be a symmetric circle with $\Lambda_{\gamma_0}^n=\Lambda(0)$. Then $\gamma_0$ can be extended to a smooth horizontal family of symmetric circles $(\gamma_t),\;t\in[0,1]$, such that for all $t$ we have $\Lambda_{\gamma_t}^n=\Lambda(t)$.
\end{lemma}
\begin{proof}
	Let $(\beta_t),\;t\in[0,1]$, be a smooth family of symmetric circles such that $\beta_0=\gamma_0$ and $\Lambda_{\beta_t}^n=\Lambda(t)$ for $t\in[0,1]$. Denote by $b_t$ the image of $\tilde{\beta}_t$ inside $\tilde{M}^1$. By Proposition~\ref{foliation}, the circle $b_t$ is transverse to $\tilde{F}$ for all $t$.  It follows that for every $t$ there is a unique section $\xi_t$ of the pullback bundle $\tilde{\beta}_t^*Tb_t$, such that for every $u\in S^1$ the tangent vector
	\begin{align*}
		v_t(u):=\pderiv[\tilde{\beta}]{t}(t,u)+\xi_t(u)\in T_{\tilde{\beta}_t(u)}\tilde{M}^1
	\end{align*}
	is tangent to $\tilde{F}$. Note that $\xi_t$ is $O_1(\mathbb{R})$ invariant. For $(t,u)\in[0,1]\times S^1$, set
	\begin{align*}
		\eta_t(u)=d(\tilde{\beta}_t)^{-1}_{\tilde{\beta}_t(u)}(\xi_t(u)),
	\end{align*}
	and note that $(\eta_t)$ is a smooth family of vector fields on $S^1$, which is $O_1(\mathbb{R})$ invariant. Let $(\varphi_t:S^1\to S^1)$ be the family of diffeomorphisms which satisfies
	\begin{align*}
		\varphi_0=id,\qquad\forall(t,u)\in[0,1]\times S^1\quad\deriv{t}\varphi_t(u)=\eta_t(\varphi_t(u)).
	\end{align*}
	For every $t$, the diffeomorphism $\varphi_t$ is $O_1(\mathbb{R})$ equivariant. Define a family of embeddings $(\tilde{\gamma}_t:S^1\to \tilde{M}^1)$ by
	\begin{align*}
	\tilde{\gamma}_t=\tilde{\beta_t}\circ\varphi_t.
	\end{align*}
	Note that since $\varphi_0=id$, the embedding $\tilde{\gamma}_0$ obtained by this definition indeed coincides with the strict transform of the given symmetric circle $\gamma_0$. Taking time derivative, we see that for a fixed $u\in S^1$
	\begin{align*}
		\deriv{t}\tilde{\gamma}_t(u)&=\pderiv[\tilde{\beta}]{t}(t,\varphi_t(u))+d\tilde{\beta}_{t,\varphi_t(u)}\left(\deriv{t}\varphi_t(u)\right)\\
		&=\pderiv[\tilde{\beta}]{t}(t,\varphi_t(u))+d\tilde{\beta}_{t,\varphi_t(u)}(\eta_t(\varphi_t(u)))\\
		&=\pderiv[\tilde{\beta}]{t}(t,\varphi_t(u))+\xi_t(\varphi_t(u))\\
		&=v_t(\varphi_t(u)),
	\end{align*}
	where by construction, the latter is tangent to $\tilde{F}$. Set $\gamma_t=\pi\circ\tilde{\gamma}_t$. Since $\tilde{\gamma}_t$ is $O_1(\mathbb{R})$ equivariant, every $\gamma_t$ is a symmetric circle. By the above calculations, the family $(\gamma_t)$ satisfies the desired properties.
\end{proof}

\section{Geodesics}

\subsection{Pseudo-Hamiltonian vector fields}
\label{PHVF}

The following will be useful for the proofs of Theorems~\ref{initial} and~\ref{boundary}.
\begin{dfn}
Let $\gamma$ be a symmetric circle with $\Lambda_\gamma^n\in\pospathlag{n}$. For $u\in S^1$, we let $l_\gamma(u)$ denote the leaf of $\tilde{F}$ which contains $\tilde{\gamma}(u)$. An open subset $U\subset\tilde{M}^1$ is \emph{uniquely projected onto $\gamma$} if
\begin{enumerate}[label=(\alph*)]
\item The strict transform $\tilde{\gamma}$ is contained in $U$,
\item For $q\in U$ there is a unique $u\in S^1$ such that $q\in l_\gamma(u)$,
\item\label{UPc} The projection $P:U\to\Lambda_\gamma^1$ which maps every point in $l_\gamma(u)$ to $\gamma(u)$ is a smooth submersion.
\end{enumerate}
Whenever discussing an open subset $U$ which is uniquely projected onto $\gamma$, we will denote the projection in \ref{UPc} by $P$, and refer to it as \emph{the projection of $U$}.
\end{dfn}
\begin{dfn}
Let $U\subset \tilde{M}^1$ be an open subset and $f:U\to\mathbb{R}$ smooth. A smooth vector field $X$ on $U$ is called \emph{the pseudo-Hamiltonian vector field corresponding to $f$} if it satisfies
\begin{align*}
i_X\pi^*\omega=d f.
\end{align*}
The reason for \emph{pseudo} is that $\pi^*\omega$ is not a symplectic form, as it vanishes on all exceptional divisors. Since the complement of the union of all exceptional divisors is dense in $\tilde{M}^1$, the pseudo-Hamiltonian vector field corresponding to $f$ is unique, if it exists. Note that a general function does not necessarily have a corresponding pseudo-Hamiltonian vector field.
\end{dfn}
\begin{prop}
	\label{pHamiltonian}
	Let $\gamma$ be a symmetric circle with $\Lambda_\gamma^n\in\pospathlag{n}$, and let $h:\Lambda_\gamma^1\to\mathbb{R}$ be smooth such that
	\begin{align*}
		\left.\deriv[h]{u}\right|_{u=0}=\left.\deriv[h]{u}\right|_{u=\pi}=0.
	\end{align*}
	Let $U\subset \tilde{M}^1$ be open and uniquely projected onto $\gamma$.
	\begin{enumerate}[label=(\alph*)]
	\item\label{phama} The function $h\circ P:U\to\mathbb{R}$ has a smooth corresponding pseudo-Hamiltonian vector field $X$, which is tangent to $\tilde{F}$.
	\item\label{phamb} For any $u_0\ne0,\pi$, and $q\in l_\gamma(u_0)\cap U$, the vector field $X$ of \ref{phama} vanishes at $q$ if and only if $d h/d u$ vanishes at $u_0$. For $q\in l_\gamma(0)\cap U$ (or $l_\gamma(\pi)\cap U$), the vector field $X$ vanishes at $q$ if and only if $d^2h/d u^2$ vanishes at $0$ (or $\pi$).
	\end{enumerate}
\end{prop}
\begin{proof}
We start with \ref{phama}. Denote by $E_0,E_\pi,$ the exceptional divisors over $\gamma(0),\gamma(\pi),$ respectively. At any point in $U\setminus(E_0\cup E_\pi)$, the form $\pi^*\omega$ is nondegenerate. The function $h\circ P|_{U\setminus(E_0\cup E_\pi)}$ thus has a corresponding Hamiltonian vector field $X$, and we need to show that $X$ can be extended smoothly to the whole of $U$.

Let $V$ be an open polar neighborhood over $\gamma(0)$ with polar coordinates $r,\theta$. Note that
\begin{align*}
d(h\circ P)=d P^*h=P^*d h,
\end{align*}
and since $d h$ vanishes at $\gamma(0)$ by assumption, $d(h\circ P)$ vanishes on $\{r=0\}\subset V\cap U$. Hence there exists a smooth 1-form $\chi$ defined on $V\cap U$ such that
\begin{align}
\label{1form}
d(h\circ P)=r\cdot\chi
\end{align}
along $V\cap U$.

We study the behavior of $\pi^*\omega$ close to $E_0$. For that, we first express $\omega$ close to $\gamma(0)$ in terms of $x$ and $y$, the real coordinates such that $z=x+iy$. Using our notation from the proof of Proposition~\ref{foliation},
\begin{align*}
\omega(\alpha\hat{\partial}_z,\beta\hat{\partial}_z) & =\omega\left(\alpha\partial_z+\frac{2z\alpha}{f'(\zeta)}\partial_\zeta,\beta\partial_z+\frac{2z\beta}{f'(\zeta)}\partial_\zeta\right)\\
& =\imaginary\left(\overline{\alpha}\beta+\overline{\left(\frac{2z\alpha}{f'(\zeta)}\right)}\frac{2z\beta}{f'(\zeta)}\right)\\
& =\left(1+\frac{4|z|^2}{|f'(\zeta)|^2}\right)\imaginary(\overline{\alpha}\beta).
\end{align*}
Set
\begin{align*}
\psi(z)=\left(1+\frac{4|z|^2}{|f'(\zeta)|^2}\right),
\end{align*}
(recall that close to $\gamma(0)$ the coordinate $\zeta$ is a smooth function of $z$) and obtain
\begin{align}
\label{sympform}
\omega=\psi(z)d x\wedge d y,
\end{align}
where $\psi\ge1$ is smooth and bounded near $\gamma(0)$. Since the projection $\pi$ is given by $(r,\theta)\mapsto(r\cos\theta,r\sin\theta)$, we have
\begin{align*}
\pi^*d x=\cos\theta d r-r\sin\theta d\theta,\quad\pi^*d y=\sin\theta d r+r\cos\theta d\theta,
\end{align*}
and
\begin{align*}
\pi^*d x\wedge d y=r(\cos^2\theta+\sin^2\theta)d r\wedge d\theta=r d r\wedge d\theta.
\end{align*}
Substituting in \eqref{sympform} yields
\begin{align}
\label{sympblowup}
\pi^*\omega=r\cdot\psi(z)d r\wedge d\theta.
\end{align}

Using \eqref{1form} and \eqref{sympblowup}, the equation
\begin{align*}
i_X\pi^*\omega=d(h\circ P)
\end{align*}
can be written as
\begin{align*}
i_Xr\cdot\psi(z)d r\wedge d\theta=r\cdot\chi,
\end{align*}
which is equivalent, anywhere outside $E_0$, to
\begin{align}
\label{reduced}
i_X\psi(z)d r\wedge d\theta=\chi.
\end{align}
Since the form $\psi(z)d r\wedge d\theta$ is nowhere degenerate, $X$ can be extended smoothly uniquely to the whole of $V\cap U$. Uniqueness implies that the extension is independent of $V$ and thus  well defined on $E_0\cap U$. Using the same argument for $E_\pi$, we conclude the existence of the pseudo-Hamiltonian vector field. It is tangent to $\tilde{F}$ as $h\circ P$ is constant along each of its leaves.

We show \ref{phamb}. It follows from \eqref{reduced} that $X=0$ if and only if $\chi=0$. For $q\in l_\gamma(u_0)$ where $u_0\ne0,\pi$, the form $\chi$ vanishing at $q$ is equivalent to $d(h\circ P)$ vanishing at $q$, and hence to $d h/d u$ vanishing at $u_0$ (since $P$ is a submersion). For $q\in l_\gamma(0)$ (or $l_\gamma(\pi)$), the vanishing of $\chi$ at $q$ is equivalent to $d^2h/d u^2$ vanishing at $0$ (or $\pi$).\qedhere
\end{proof}

\subsection{Pseudo-Hamiltonian flow and geodesics}
\label{PHFG}

The following proposition establishes the relation between geodesics in $\pospathlag{n}$ and pseudo-Hamiltonian flow in $\tilde{M}^1$. It will be the main argument in the proof of Theorem~\ref{initial}.
\begin{prop}
\label{gflow}
 Let $\gamma_0$ be a symmetric circle with $\Lambda_{\gamma_0}^n\in\pospathlag{n}$, let $h_0:\Lambda_{\gamma_0}^n\to\mathbb{R}$ be smooth and invariant under the $O_n(\mathbb{R})$ action and recall the embedding
\begin{align*}
i:\Lambda_{\gamma_0}^1  \to\Lambda_{\gamma_0}^n,\quad
(z,\zeta)  \mapsto(z,0,\ldots,0,\zeta).
\end{align*}
Let $U\subset\tilde{M}^1$ be open and uniquely projected onto $\gamma_0$.
\begin{enumerate}[label=(\alph*)]
 \item\label{gflowa} The function $h_0\circ i\circ P:U\to\mathbb{R}$ has a smooth corresponding pseudo-Hamiltonian vector field $X$ along $U$, which is tangent to $\tilde{F}$ and invariant under the $O_1(\mathbb{R})$ action.
\item\label{gflowb} Let $(\tilde{\gamma}_t:S^1\to U)$ be a smooth family of embeddings which extends  $\tilde{\gamma}_0$ and satisfies
\begin{align}
	\label{xprecise}
\pderiv[\tilde{\gamma}]{t}(t,u)=X(\tilde{\gamma}_t(u))
\end{align}
for all $t,u$. Then for every $t$ the composition $\gamma_t=\pi\circ\tilde{\gamma}_t$ is a symmetric circle with $\Lambda_{\gamma_t}^n\in\pospathlag{n}$, and the Lagrangian path $(\Lambda_{\gamma_t}^n)$ is a geodesic with
\begin{align*}
\left.\deriv{t}\right|_{t=0}\Lambda_{\gamma_t}^n=d h_0.
\end{align*}
\item\label{gflowc}Let $(\alpha_t)$ be a smooth horizontal family of symmetric circles, such that for every $t$ the image of the strict transform $\tilde{\alpha}_t$ is contained in $U$. Assume $\alpha_0=\gamma_0$, and that the Lagrangian path $(\Lambda_{\alpha_t}^n)$ is a geodesic with derivative
\begin{align*}
\left.\deriv{t}\right|_{t=0}\Lambda_{\alpha_t}^n=d h_0.
\end{align*}
Then for every $t$ we have $\alpha_t=\gamma_t$.
\end{enumerate}
\end{prop}
\begin{proof}
We begin with \ref{gflowa}. Since $h_0$ is invariant under the $O_n(\mathbb{R})$ action, it follows that
\begin{align*}
h_0\circ i(\gamma_0(-u))=h_0\circ i(\gamma_0(u))
\end{align*}
for $u\in S^1$. Hence,
\begin{align*}
\left.\deriv{u}\right|_{u=0}h_0\circ i=\left.\deriv{u}\right|_{u=\pi}h_0\circ i=0.
\end{align*}
Proposition~\ref{pHamiltonian} \ref{phama} thus implies that $h_0\circ i\circ P$ has a smooth corresponding pseudo-Hamiltonian vector field which is tangent to  $\tilde{F}$. It is invariant under the $O_1(\mathbb{R})$ action as $\omega$ and $h_0\circ i\circ P$ are.

We show \ref{gflowb}. Since $\gamma_0$ is an embedding, $\tilde{\gamma}_0$ intersects the exceptional divisors (which will be referred to as $E_0,E_\pi$) transversally, by Lemma~\ref{liftcurve}. Since $X$ is tangent to $E_0,E_\pi$, the embedding $\tilde{\gamma}_t$ is transverse to the exceptional divisors for all $t$. Setting $\gamma_t:=\pi\circ\tilde{\gamma}_t$ we thus obtain a smooth family of embeddings $(\gamma_t:S^1\to M^1)$. Since $\Lambda_{\gamma_0}^1$ and $X$ are invariant under the $O_1(\mathbb{R})$ action, $\gamma_t$ is a symmetric circle for all $t$. By Lemma~\ref{lagrangian} and Lemma~\ref{regular} \ref{regularb}, the family $(\gamma_t)$ induces a Lagrangian path $(\Lambda_{\gamma_t}^n)$. As $\tilde{\gamma}_0$ intersects all leaves of $\tilde{F}$ transversally, and $X$ is tangent to those leaves, every $\tilde{\gamma}_t$ intersects the leaves transversally, and $(\Lambda_{\gamma_t}^n)$ is a path of positive Lagrangians by Proposition~\ref{foliation}.

Let $\Psi$ denote the lift of $(\Lambda_{\gamma_t}^n)$ induced by $(\gamma_t)$ as in Remark~\ref{fields}. Note that $\Psi$ is horizontal by~\eqref{xprecise}, as $X$ is tangent to $\tilde{F}$. For every $t$ let
\begin{align*}
	\varphi_t:\Lambda_{\gamma_0}^n\to\Lambda_{\gamma_t}^n
\end{align*}
denote the diffeomorphism corresponding to $\Psi$. That is, $\varphi_t=\Psi_t\circ\Psi_0^{-1}$. Define $h_t:\Lambda_{\gamma_t}^n\to\mathbb{R}$ by
\begin{align*}
h_t=h_0\circ\varphi_t^{-1}.
\end{align*}
Note that for every $t$ the function $h_t$ is $O_n(\mathbb{R})$ invariant as $\varphi_t$ is $O_n(\mathbb{R})$ equivariant. We show that the time derivative of $\Lambda_{\gamma_t}^n$ is given by
\begin{align*}
\deriv{t}\Lambda_{\gamma_t}^n=d h_t,
\end{align*}
and by that complete the proof. Recall the embedding $I$ of Remark~\ref{capitali}. For some $t$ and $u\ne0,\pi,$ let $X_1(\gamma_t(u)),\ldots,X_n(\gamma_t(u))\in T_{I(\gamma_t(u))}\Lambda_{\gamma_t}^n$ as in Remark~\ref{frame}. Note that by construction we have
\begin{align*}
I(\gamma_t(u))=\varphi_t(I(\gamma_0(u))).
\end{align*}
Hence,
\begin{align*}
h_t(I(\gamma_t(u)))=h_0\circ I(\gamma_0(u))=h_0\circ i(\gamma_0(u))=h_0\circ i\circ P(\tilde{\gamma}_t(u)).
\end{align*}
Using this together with \eqref{xprecise} and $X$ corresponding to $h_0\circ i\circ P$, we obtain
\begin{align*}
d (h_t)_{I(\gamma_t(u))}(X_1) & =d(h_0\circ i\circ P)_{\tilde{\gamma}_t(u)}\left(\pderiv[\tilde{\gamma}]{u}(t,u)\right)\\
& =\pi^*\omega\left(\pderiv[\tilde{\gamma}]{t}(t,u),\pderiv[\tilde{\gamma}]{u}(t,u)\right)\\
& =\omega\left(\pderiv[\gamma]{t}(t,u),\pderiv[\gamma]{u}(t,u)\right).
\end{align*}
Since $I^*\omega=\omega$, we conclude
\begin{align}
\label{derivative1}
d(h_t)_{I(\gamma_t(u))}(X_1)=\omega\left(d I_{\gamma_t(u)}\left(\pderiv[\gamma]{t}\right),X_1\right)=\deriv{t}\Lambda_{\gamma_t}^n(X_1).
\end{align}
For $j=2,\ldots,n,$ we have
\begin{align}
\label{derivative2}
d h_t(X_j)=0=\deriv{t}\Lambda_{\gamma_t}^n(X_j),
\end{align}
where the left hand equality follows from $h_t$ being invariant under the $O_n(\mathbb{R})$ action, and the right hand one from Proposition~\ref{paths}. It follows from \eqref{derivative1} and \eqref{derivative2} that $d\Lambda_{\gamma_t}^n/d t$ and $d h_t$ agree on $I(\Lambda_{\gamma_t}^1)$. By invariance under the $SO_n(\mathbb{R})$ action, the forms agree on the whole of $\Lambda_{\gamma_t}^n$.

We prove \ref{gflowc}. It is enough to show that the family $(\tilde{\alpha}_t)$ satisfies
\begin{align*}
\pderiv[\tilde{\alpha}]{t}(t,u)=X(\tilde{\alpha}_t(u))
\end{align*}
for all $t,u$. Let $\tilde{a}_t\subset\tilde{M}^1$ stand for the image of $\tilde{\alpha}_t$, and set
\begin{align*}
	\tilde{h}_t=h_0\circ i\circ P|_{\tilde{a}_t}:\tilde{a}_t\to\mathbb{R}.
\end{align*}
Let $\Psi$ denote the lift of the Lagrangian path $(\Lambda_{\alpha_t}^n)$, induced by $(\alpha_t)$ as in Remark~\ref{fields}, and set
\begin{align*}
h_t=h_0\circ\Psi_0\circ\Psi_t^{-1}:\Lambda_{\alpha_t}^n\to\mathbb{R}.
\end{align*}
Since $(\Lambda_{\alpha_t}^n)$ is a geodesic, and the lift $\Psi$ is horizontal (by Proposition~\ref{foliation}), for every $u,t$, we have
\begin{align*}
	\pi^*\omega\left(\pderiv[\tilde{\alpha}]{t}(t,u),\pderiv[\tilde{\alpha}]{u}(t,u)\right)&=\omega\left(\pderiv[\alpha]{t}(t,u),\pderiv[\alpha]{u}(t,u)\right)\\
	&=\omega\left(d I\left(\pderiv[\alpha]{t}(t,u)\right),d I\left(\pderiv[\alpha]{u}(t,u)\right)\right)\\
	&=d h_t\left(d I\left(\pderiv[\alpha]{u}(t,u)\right)\right)\\
	&=d\tilde{h}_t\left(\pderiv[\tilde{\alpha}]{u}(t,u)\right)\\
	&=d(h_0\circ i\circ P)\left(\pderiv[\tilde{\alpha}]{u}(t,u)\right)\\
	&=\pi^*\omega\left(X(\tilde{\alpha}_t(u)),\pderiv[\tilde{\alpha}]{u}(t,u)\right).
\end{align*}
Since both vectors $X(\tilde{\alpha}_t(u))$ and $\partial\tilde{\alpha}/\partial t(t,u)$ are tangent to $\tilde{F}$, whereas the vector $\partial\tilde{\alpha}/\partial u(t,u)$ is not, the above calculation verifies that
\begin{align*}
	\pderiv[\tilde{\alpha}]{t}(t,u)=X(\tilde{\alpha}_t(u)),
\end{align*}
as stated.\qedhere
\end{proof}

\begin{proof}[Proof of Theorem~\ref{initial}]
We show that for any smooth $O_n(\mathbb{R})$ invariant function $h_0:\Lambda_0\to\mathbb{R}$ there is a geodesic $\Lambda:[0,\epsilon]\to\pospathlag{n}$ for some $\epsilon>0$, such that
\begin{align*}
\Lambda(0)=\Lambda_0,\quad\left.\deriv{t}\right|_{t=0}\Lambda(t)=d h_0.
\end{align*}
The existence part of the theorem, as well as the promised equality of Remark~\ref{promisedeq}, then follows from Proposition~\ref{paths}.

Let $h_0:\Lambda_0\to\mathbb{R}$ be as stated above, let $\gamma_0$ be a symmetric circle with $\Lambda_{\gamma_0}^n=\Lambda_0$, and let $g$ be any Riemannian metric on $\tilde{M}^1$. For $u\in S^1$, we use the notation $l_\gamma(u)$ as above. For $\delta>0$, denote by $l_\gamma(u,\delta)$ the metric ball in  $l_\gamma(u)$ of radius $\delta$ around $\tilde{\gamma}_0(u)$. Take $\delta$ so small that the following hold:
\begin{enumerate}[label=(\alph*)]
\item The distance between the image of $\tilde{\gamma}_0$ and $H$, the set of singular points of $\tilde{F}$, is greater than $\delta$.
\item For $u\in S^1$ the ball $l_\gamma(u,\delta)$ is diffeomorphic to an interval.
\item For $u\neq u'\in S^1$ the intervals $l_\gamma(u,\delta),l_\gamma(u',\delta),$ are disjoint.
\end{enumerate}
Such a $\delta$ exists as $\tilde{\gamma}_0$ is transverse to all leaves of $\tilde{F}$. For the same reason the set
\begin{align*}
U=\bigcup_{u\in S^1}l_\gamma(u,\delta)
\end{align*}
is an open subset of $\tilde{M}^1$. Note that the projection $P:U\to\Lambda_{\gamma_0}^1$ that maps every point in $l_\gamma(u,\delta)$ to $\gamma_0(u)$ is a smooth submersion, and so $U$ is uniquely projected onto $\gamma_0$. By Proposition~\ref{gflow} \ref{gflowa} the function $h_0\circ i\circ P$ has a corresponding pseudo-Hamiltonian vector field $X$ along $U$. Let $x$ denote the flow of $X$. By compactness of $\tilde{\Lambda}_{\gamma_0}^1$ there is some $\epsilon>0$ such that for all $u\in S^1$ and $0\leq t\leq\epsilon$ we have $x_t(\tilde{\gamma}_0(u))\in U$. Thus we define the family $(\tilde{\gamma}_t:S^1\to\tilde{M}^1),\;0\leq t\leq\epsilon$, by $\tilde{\gamma}_t(u)=x_t(\tilde{\gamma}_0(u))$, and let $\gamma_t:=\pi\circ\tilde{\gamma}_t$ for all $t$. By Proposition~\ref{gflow} \ref{gflowb}, the desired geodesic is obtained by setting $\Lambda(t)=\Lambda_{\gamma_t}^n$.

We show uniqueness. Let $\Lambda':[0,\epsilon]\to\pospathlag{n}$ be another geodesic which satisfies the same initial value condition. It is enough to verify that for some $0<\epsilon'\leq\epsilon$ the equality $\Lambda'|_{[0,\epsilon']}=\Lambda|_{[0,\epsilon']}$ holds. Take a smooth family of circles $(\alpha_t)$ such that $\Lambda_{\alpha_t}^n=\Lambda'(t)$ for $t\in[0,\epsilon]$. By Lemma~\ref{horizontalcircles}, we may assume $\alpha_0=\gamma_0$, and that the family $(\alpha_t)$ is horizontal. By compactness, there exists $\epsilon'>0$ such that for $t\in[0,\epsilon']$, the image of $\tilde{\alpha}_t$ is contained in $U$. It follows from Proposition~\ref{gflow} \ref{gflowc} that $\alpha_t=\gamma_t$ for $t\in[0,\epsilon']$.
\end{proof}

\subsection{Cylindrical foliations}
\label{CF}

The following notion is helpful for solving the boundary value problem:
\begin{dfn}
Let $\gamma$ be a symmetric circle with $\Lambda_\gamma^n\in\pospathlag{n}$. We say that $\tilde{F}$ is \emph{cylindrical over $\gamma$} if
\begin{enumerate}[label=(\alph*)]
\item $l_\gamma(u)$ is diffeomorphic to $\mathbb{R}$ for all $u\in S^1$,
\item $l_\gamma(u)\neq l_\gamma(u')$ whenever $u\neq u'$.
\end{enumerate}
\end{dfn}

\begin{prop}
\label{solutionwhencylindrical}
Let $\mathcal{O}\subset\pospathlag{n}$ be a Hamiltonian isotopy class, let $\Lambda_0,\Lambda_1\in\mathcal{O}$, and let $\gamma_0$ be a symmetric circle with $\Lambda_{\gamma_0}^n=\Lambda_0$. If $\tilde{F}$ is cylindrical over $\gamma_0$, then there is a unique geodesic connecting $\Lambda_0$ and $\Lambda_1$.
\end{prop}

\begin{proof}
Set
\begin{align}
\label{capitalu}
U=\bigcup_{u\in S^1}l_{\gamma_0}(u),
\end{align}
and note that by assumption $U$ projects uniquely onto $\gamma_0$. By Lemma~\ref{horizontalcircles}, since $\Lambda_0,\Lambda_1$, are Hamiltonian isotopic in $\pospathlag{n}$, there is a smooth horizontal family of symmetric circles $(\beta_t),\;t\in[0,1]$, which satisfies
\begin{enumerate}[label=(\alph*)]
\item $\beta_0=\gamma_0$,
\item $\Lambda_{\beta_1}^n=\Lambda_1$.
\end{enumerate}
Note that for all $u\in S^1$ we have $\tilde{\beta}_1(u)\in l_{\gamma_0}(u)$, as $(\beta_t)$ is horizontal.

Define a function $k:\Lambda_{\gamma_0}^1\to\mathbb{R}$ by
\begin{align}
\label{smallk}
\gamma_0(u)\mapsto\cos(u),
\end{align}
and let $P:U\to\Lambda_{\gamma_0}^1$ as above. Proposition~\ref{pHamiltonian} implies that $k\circ P$ has a corresponding nonvanishing pseudo-Hamiltonian vector field $X$ along $U$. $X$ is invariant under the $O_1(\mathbb{R})$ action, since $k$ is.

Let $x$ denote the flow of $X$, and define $s:\Lambda_{\gamma_0}^1\to\mathbb{R}$ by
\begin{align}
\label{flowtime}
x_{s(\gamma_0(u))}(\tilde{\gamma}_0(u))=\tilde{\beta}_1(u).
\end{align}
Since $X$ does not vanish and both  $\tilde{\gamma}_0,\tilde{\beta}_1$, are transverse to the trajectories of $X$, it follows from the inverse function theorem that $s$ is well defined and smooth. Furthermore, since $X,\tilde{\gamma}_0$, and $\tilde{\beta}_1$ are invariant under the $O_1(\mathbb{R})$ action, so is $s$.

Define another vector field on $U$ by
\begin{align}
\label{hiY}
Y=(s\circ P)\cdot X,
\end{align}
denote by $y$ the flow of $Y$, and note that by definition of $s$
\begin{align*}
\forall u\in S^1\quad y_1(\tilde{\gamma}_0(u))=x_{s(\gamma_0(u))}(\tilde{\gamma}_0(u))=\tilde{\beta}_1(u).
\end{align*}
We show that there is a smooth $O_1(\mathbb{R})$ invariant $h:\Lambda_{\gamma_0}^1\to\mathbb{R}$ such that $Y$ is the pseudo-Hamiltonian vector field that corresponds to $h\circ P$. By direct calculation
\begin{align*}
i_Y\omega=(s\circ P)i_X\omega=(s\circ P)d(k\circ P)=P^*(sdk).
\end{align*}
Since both $s$ and $k$ are even functions of $u$, the function $s\cdot d k/d u$ is odd. Hence, there is $h:\Lambda_{\gamma_0}^1\to\mathbb{R}$ which is $O_1(\mathbb{R})$ invariant and satisfies
\begin{align}
\label{smallh}
d h=sdk.
\end{align}
Now
\begin{align*}
i_Y\omega=P^*d h=d(P^*h)=d(h\circ P),
\end{align*}
as desired.

Embed $\Lambda_{\gamma_0}^1$ in $\Lambda_{\gamma_0}^n=\Lambda_0$ by $I$, extend $h$ to $\hat{h}:\Lambda_0\to\mathbb{R}$ which is $O_n(\mathbb{R})$ invariant, and note that by Lemma \ref{smoothonsphere} $\hat{h}$ is smooth. Extend $\tilde{\gamma}_0$ to a smooth family $(\tilde{\gamma}_t:S^1\to \tilde{M}^1),\;t\in[0,1]$, by $\partial\tilde{\gamma}/\partial t=Y$, and by Proposition~\ref{gflow} \ref{gflowb} the induced path $(\Lambda_{\gamma_t}^n)$ is a geodesic. By construction, $\tilde{\gamma}_1=\tilde{\beta}_1$. Hence, $\Lambda_{\gamma_1}^n=\Lambda_{\beta_1}^n=\Lambda_1$, and existence of a geodesic is proved.

Let $\Lambda':[0,a]\to\pospathlag{n}$ be another geodesic with $\Lambda'(0)=\Lambda_0,\;\Lambda'(a)=\Lambda_1$. Reparametrizing if necessary, we may assume $a=1$. Let $\hat{h}':\Lambda_0\to\mathbb{R}$ satisfy
\begin{align*}
	\left.\deriv{t}\right|_{t=0}\Lambda'(t)=d\hat{h}'.
\end{align*}
By proposition~\ref{paths}, $\hat{h}'$ is $O_n(\mathbb{R})$ invariant, and by Proposition~\ref{gflow} \ref{gflowa}, the function $\hat{h}'\circ i\circ P:U\to\mathbb{R}$ has a corresponding pseudo-Hamiltonian vector field $W$. Let $w$ denote the flow of $W$. Let $(\alpha_t)$ be a family of smooth circles with $\Lambda_{\alpha_t}^n=\Lambda'(t)$ for $t\in[0,1]$. By Lemma~\ref{horizontalcircles}, we may assume $\alpha_0=\gamma_0$, and that $(\alpha_t)$ is horizontal. It follows that for every $u$ we have $\tilde{\alpha}_1(u)=\tilde{\gamma}_1(u)$.

By Proposition~\ref{gflow} \ref{gflowc}, for every $u\in S^1$ we have
\begin{align*}
w_1(\tilde{\alpha}_0(u))=\tilde{\alpha}_1(u),
\end{align*}
which can be written as
\begin{align}
	\label{agree1}
	w_1(\tilde{\alpha}_0(u))=y_1(\tilde{\alpha}_0(u)).
\end{align}

Set $h'=\hat{h}'\circ I:\Lambda_{\gamma_0}^1\to\mathbb{R}$. Define a function $r:\Lambda_{\gamma_0}^1\to\mathbb{R}$ by
\begin{align*}
	d h'=r\cdot d k,
\end{align*}
and note that
\begin{gather*}
	i_W\pi^*\omega=d(P^*h')=P^*(d h')=P^*(r\cdot d k)=r\circ P\cdot d(k\circ P)\\=r\circ P\cdot i_X\pi^*\omega,
\end{gather*}
which yields
\begin{align}
	\label{hiW}
	W=r\circ P\cdot X.
\end{align}
Putting \eqref{hiY}, \eqref{hiW} and \eqref{agree1} together, for every $u$ we have
\begin{align}
	\label{agree2}
	x_{r(\alpha_0(u))}(\tilde{\alpha}_0(u))=x_{s(\alpha_0(u))}(\tilde{\alpha}_0(u)).
\end{align}

For every $u$, the map
\begin{align*}
	\varphi_u:[0,1]\to U,\quad t\mapsto x_t(\tilde{\alpha}_0(u)),
\end{align*}
is injective since $\tilde{F}$ is cylindrical over $\alpha_0$ and $X$ is nonvanishing. Hence, it follows from~\eqref{agree2} that $r=s$, which implies $W=Y$. The geodesic $\Lambda'$ thus coincides with the above geodesic $(\Lambda_{\gamma_t}^n)$.
\end{proof}

\begin{prop}
\label{cylindrical}
If $n\geq2$, then $\tilde{F}$ is cylindrical over any symmetric circle $\gamma$ whose corresponding Lagrangian $\Lambda_\gamma^n$ is positive.
\end{prop}

\begin{proof}[Proof of Theorem~\ref{boundary}]
Follows immediately from Proposition~\ref{solutionwhencylindrical} and Proposition~\ref{cylindrical}.
\end{proof}

\begin{proof}[Proof of Proposition~\ref{cylindrical}]
By Proposition~\ref{foliation}, along an exceptional divisor $E$ the foliation $\tilde{F}$ is tangent to $E$, and has $n$ singular points. Furthermore, for every symmetric circle $\gamma$, the strict transform $\tilde{\gamma}$ intersects any exceptional divisor at most once. Thus it suffices to focus on leaves of $\tilde{F}$ outside the exceptional divisors in $\tilde{M}^1$, or in other words, study the induced foliation $F$ of $M^1\setminus\{z=0\}$. The projection
\begin{align*}
	\rho:M^1\setminus\{z=0\}\to\mathbb{C}\setminus\{f=0\},\quad(z,\zeta)\mapsto\zeta,
\end{align*}
is a smooth covering map. Since the foliation $F$ is invariant under the  nontrivial deck transformation, it induces a well defined foliation of $\mathbb{C}\setminus\{f=0\}$, denoted by $\underline{F}$.

For a symmetric circle $\gamma$, denote by $\underline{\gamma}$ its image in $\mathbb{C}$ under $(z,\zeta)\mapsto\zeta$. Note that if $\tilde{\gamma}$ is transverse to $\tilde{F}$, then $\underline{\gamma}$ is transverse to $\underline{F}$. If $\tilde{F}$ has a closed leaf, then so does $\underline{F}$. If a leaf of $\tilde{F}$ intersects the strict transform $\tilde{\gamma}$ at two different points, then one of the following holds:
\begin{enumerate}[label=(\alph*)]
	\item The foliation $\underline{F}$ has a closed leaf.
	\item There is a leaf of $\underline{F}$ which intersects $\underline{\gamma}$ at two different points.
\end{enumerate}
Hence, it is enough to show that the foliation $\underline{F}$ has no closed leaves, and that for every symmetric circle $\gamma$ with $\Lambda_\gamma^n$ positive, no leaf of $\underline{F}$ intersects $\underline{\gamma}$ at two different points.

Consider first the case $n=2$. As follows directly from the construction, the leaves of $\underline{F}$ in this case are straight lines, parallel to the imaginary axis, and the proposition holds.

Now let $n\geq3$. The following arguments are similar to those used by Shapere-Vafa in \cite{shapere-vafa}. Assume that $\underline{F}$ has a closed leaf, parametrized by $\zeta:[0,1]\to\mathbb{C}$. By the theorem of the turning tangents, and reversing orientation if necessary, we assume $\zeta$ has turning number $+1$. Let $z:[0,1]\to\mathbb{C}$ be continuous and satisfy $z^2=f(\zeta)$, and denote by $w$ the winding number of $z^{n-2}$ (which is a well defined half integer). Since
\begin{align}
\label{zeroinside}
z=\mathrm{constant}\cdot\prod_j\sqrt{\zeta-\zeta_j},
\end{align}
where the $\zeta_j$s are the roots of $f$, and $\zeta$ is parametrized counterclockwise, $w$ is nonnegative. On the other hand, as follows from the construction of the foliation, the argument of $(z(x))^{n-2}\cdot\zeta'(x)$ is constant along $[0,1]$. Since the winding number of $\zeta'$ is equal to $+1$, it follows that $w=-1$, which leads to a contradiction, and so $\underline{F}$ has no closed leaves.

Let $\gamma$ be a symmetric circle with $\Lambda_\gamma^n$ positive, and assume there are two distinct points $p,p',$ on $\underline{\gamma}$ which share the same leaf of $\underline{F}$, denoted by $\underline{l}$. We may assume that $\underline{l}$ does not intersect $\underline{\gamma}$ at any point between $p$ and $p'$, so $\underline{\gamma}$ and $\underline{l}$ form a simple loop which is smooth anywhere except $p$ and $p'$. Parametrize this loop by $\alpha:[0,1]\to\mathbb{C}$, that goes from $p$ to $p'$ along $\underline{\gamma}$ and back to $p$ along $\underline{l}$. Switching the two points if needed, we also assume that $\alpha$ has turning number $+1$.

We claim that there is a zero of $f$ inside $\alpha$. Otherwise, its image is contained in a simply connected domain $D$, in which $f$ has a holomorphic square root, $z$. It follows from holomorphicity and simply connectedness that
\begin{align*}
\int_\alpha z^{n-2}d\zeta=0.
\end{align*}
But along $\underline{l}$ the product $(z(x))^{n-2}\cdot\alpha'(x)$ is pure imaginary,  whereas along $\underline{\gamma}$ it has a nonvanishing real part. The last equation is thus impossible, and indeed, there is at least one zero of $f$ inside $\alpha$.

Once again, let $z:[0,1]\to\mathbb{C}$ be continuous and satisfy $z^2=f(\alpha)$. Let $\delta_\gamma$ and $\delta_l$ denote the changes in the argument of $z^{n-2}$ along $\underline{\gamma}$ and $\underline{l}$ respectively. Similarly, let $\Delta_\gamma$ and $\Delta_l$ denote the changes in the argument of $\alpha'$ along $\underline{\gamma}$ and $\underline{l}$ respectively. The total change in the argument of $\alpha'$, including the exterior angles at the two singular points, is $+2\pi$ by the theorem of turning tangents. Since each singular point contributes less than $\pi$, we have
\begin{align}
\label{tchangealpha}
\Delta_\gamma+\Delta_l>0.
\end{align}
Since the argument of $(z(x))^{n-2}\cdot\alpha'(x)$ is constant along $\underline{l}$, we have
\begin{align}
\label{zalphaleaf}
\delta_l=-\Delta_l.
\end{align}
Since $\Lambda_\gamma^n$ is positive, along $\underline{\gamma}$ we have
\begin{align*}
-\frac{\pi}{2}<\arg((z(x))^{n-2}\cdot\alpha'(x))<\frac{\pi}{2}.
\end{align*}
Hence,
\begin{align}
\label{zalphagamma}
\delta_\gamma<-\Delta_\gamma+\pi.
\end{align}
Putting \eqref{tchangealpha},\eqref{zalphaleaf} and \eqref{zalphagamma} together yields
\begin{align*}
\delta_\gamma+\delta_l<\pi,
\end{align*}
which means that the winding number of $z$ is $<1/2$. But since there is  a zero of $f$ inside $\alpha$, \eqref{zeroinside} implies that $z$ has a winding number $\geq1/2$, which leads to a contradiction.
\end{proof}

\subsection{The metric on a Hamiltonian isotopy class}
\label{metric}
Let $n\geq2$, and let $\mathcal{O}\subset\pospathlag{n}$ be a Hamiltonian isotopy class. For $\Lambda\in\mathcal{O}$, think of the tangent space $T_\Lambda\pathlag{n}$ as the set
\[
\mathcal{H}_\Lambda=\left\{h\in C^\infty(\Lambda)\left|h\;\mathrm{is}\;O_n(\mathbb{R})\;\mathrm{invariant,}\int_{\Lambda}h\real\Omega=0\right.\right\},
\]
and recall the Riemannian metric
\[
	\Upsilon(h,k)=\int_\Lambda hk\real\Omega.
\]
Note that $(\mathcal{H}_\Lambda,\Upsilon)$ is isometric to a linear subspace of $C^\infty([0,1])$ with the $L^2$ metric.
\begin{lemma}
	\label{natdiff}
	Horizontal lifts give rise to a family of diffeomorphisms
	\[
	(\varphi_{\Lambda_0,\Lambda_1}:\Lambda_0\to\Lambda_1)_{\Lambda_0,\Lambda_1\in\mathcal{O}},
	\]
	which are all $O_n(\mathbb{R})$ equivariant and $\real\Omega$ preserving. This family is natural in the sense that for $\Lambda_0,\Lambda_1,\Lambda_2\in\mathcal{O}$ we have
	\[
	\varphi_{\Lambda_0,\Lambda_2}=\varphi_{\Lambda_1,\Lambda_2}\circ\varphi_{\Lambda_0,\Lambda_1}.
	\]
\end{lemma}
\begin{proof}
	Let $\Lambda_0,\Lambda_1\in\mathcal{O}$, and let $(\gamma_t),\;t\in[0,1],$ be a smooth horizontal family of symmetric circles with $\Lambda_{\gamma_0}^n=\Lambda_0,\Lambda_{\gamma_1}^n=\Lambda_1$. Let $\Psi$ denote the lift of the Lagrangian path $(\Lambda_{\gamma_t}^n)$ induced by $(\gamma_t)$ as in Remark~\ref{fields}, and set
	\[
	\varphi_{\Lambda_0,\Lambda_1}=\Psi_1\circ\Psi_0^{-1}.
	\]
	The diffeomorphism $\varphi_{\Lambda_0,\Lambda_1}$ is $O_n(\mathbb{R})$ equivariant as $\Psi_0$ and $\Psi_1$ are. It is $\real\Omega$ preserving since $\Psi$ is a horizontal lift, by Proposition~\ref{foliation}.
	
	We show that $\varphi_{\Lambda_0,\Lambda_1}$ is independent of the horizontal family $(\gamma_t)$. For that, think of $M^1$ as a submanifold of $M^n$ via the embedding $I$ of Remark~\ref{capitali}. By Proposition~\ref{cylindrical}, the foliation $\tilde{F}$ is cylindrical. The restricted map
	$
	\varphi_{\Lambda_0,\Lambda_1}|_{\Lambda_0\cap M^1}
	$
	carries a point $p=(z,\zeta),z\neq0$, to the unique point in $\Lambda_1\cap M^1$ which shares a leaf of the foliation $F$ with $p$. Hence, $\varphi_{\Lambda_0,\Lambda_1}|_{\Lambda_0\cap M^1}$ is independent of $(\gamma_t)$. By $O_n(\mathbb{R})$ equivariance, so is $\varphi_{\Lambda_0,\Lambda_1}$. The same argument verifies that the family $(\varphi_{\Lambda_0,\Lambda_1})$ is natural.
\end{proof}
\begin{rem}
	\label{natdiff2}
	The family $(\varphi_{\Lambda_0,\Lambda_1})$ of Lemma~\ref{natdiff} induces a family $(\varphi_{\Lambda_0,\Lambda_1}^*:\mathcal{H}_{\Lambda_1}\to\mathcal{H}_{\Lambda_0})$ by
	\[
	\varphi_{\Lambda_0,\Lambda_1}^*(h)=h\circ\varphi_{\Lambda_0,\Lambda_1},\quad h\in\mathcal{H}_{\Lambda_1}.
	\]
	For any $\Lambda_0,\Lambda_1$, the map $\varphi_{\Lambda_0,\Lambda_1}^*$ is a well defined isometry, since $\varphi_{\Lambda_0,\Lambda_1}$ is $O_n(\mathbb{R})$ equivariant and $\real\Omega$ preserving. Hence, all tangent spaces $\{\mathcal{H}_\Lambda|\Lambda\in\mathcal{O}\}$ are naturally isometric to one another.
\end{rem}

By Theorem~\ref{initial}, every function in $\mathcal{H}_\Lambda$ is the derivative of a unique geodesic starting at $\Lambda$. Let $A_\Lambda\subset\mathcal{H}_\Lambda$ consist of all functions whose corresponding geodesics are of length $>1$, and define the exponential map
\[
\exp_{\Lambda}:A_\Lambda\to\mathcal{O}
\]
in the usual manner. By Theorem~\ref{boundary}, $\exp_\Lambda$ is a bijection.

\begin{lemma}
	\label{triangle}
	Let $\Lambda_0,\Lambda_1,\Lambda_2\in\mathcal{O}$. Let $\hat{h}_{0,1},\hat{h}_{0,2}\in A_{\Lambda_0},\hat{h}_{1,2}\in A_{\Lambda_1}$, satisfy
	\[
	\exp_{\Lambda_0}(\hat{h}_{0,1})=\Lambda_1,\quad\exp_{\Lambda_0}(\hat{h}_{0,2})=\Lambda_2,\quad\exp_{\Lambda_1}(\hat{h}_{1,2})=\Lambda_2.
	\]
	Then we have
	\[
	\hat{h}_{0,2}=\hat{h}_{0,1}+\varphi_{\Lambda_0,\Lambda_1}^*(\hat{h}_{1,2}).
	\]
\end{lemma}
\begin{proof}
	Let $\gamma_0,\gamma_1,$ be symmetric circles such that $\Lambda_{\gamma_i}^n=\Lambda_i$, and $l_{\gamma_0}(u)=l_{\gamma_1}(u)$ for $u\in S^1$. Let $\psi_{0,1}:\Lambda_{\gamma_0}^1\to\Lambda_{\gamma_1}^1$ be given by
	\[
	\psi_{0,1}=\gamma_1\circ\gamma_0^{-1}.
	\]
	The functions $\hat{h}_{0,1},\hat{h}_{0,2},\hat{h}_{1,2},$ can be obtained as in the proof of Proposition~\ref{solutionwhencylindrical}. We follow the steps in the proof, and show that the desired equality holds.
	
	The open subset $U\subset\tilde{M}^1$, defined in~\eqref{capitalu}, is the same for $\gamma_0,\gamma_1$. Define two functions
	\[
	k_0:\Lambda_{\gamma_0}^1\to\mathbb{R},\quad k_1:\Lambda_{\gamma_1}^1\to\mathbb{R},
	\]
	as in \eqref{smallk}. Since $l_{\gamma_0}(u)=l_{\gamma_1}(u)$ for every $u$, we have
	\[
	k_0=k_1\circ\psi_{0,1}=\psi_{0,1}^*k_1.
	\]
	Define functions
	\[
	s_{0,1},s_{0,2}:\Lambda_{\gamma_0}^1\to\mathbb{R},\quad s_{1,2}:\Lambda_{\gamma_1}^1\to\mathbb{R},
	\]
	as in \eqref{flowtime}, such that the function $s_{i,j}$ corresponds to a geodesic connecting $\Lambda_i$ with $\Lambda_j$. The main observation of the present proof is that by construction we have
	\[
	s_{0,2}=s_{0,1}+\psi_{0,1}^*s_{1,2}.
	\]
	Define the functions
	\[
	h_{0,1},h_{0,2}:\Lambda_{\gamma_0}^1\to\mathbb{R},\quad h_{1,2}:\Lambda_{\gamma_1}^1\to\mathbb{R},
	\]
	as in \eqref{smallh}. We have
	\begin{align*}
	d h_{0,2}&=s_{0,2}d k_0\\&=s_{0,1}d k_0+\psi_{0,1}^*(s_{1,2}d k_1)\\&=d h_{0,1}+\psi_{0,1}^*d h_{1,2}\\&=d(h_{0,1}+\psi_{0,1}^*h_{1,2}).
	\end{align*}
	The functions $\hat{h}_{i,j}$ are just $O_n(\mathbb{R})$ invariant extensions of $h_{i,j}$ to $\Lambda_i$. It follows that the desired equality holds up to some constant. Since all functions have mean value $0$ with respect to $\real\Omega$, this constant vanishes.
\end{proof}

\begin{proof}[Proof of Theorem \ref{metricspace}]
	Let $\Lambda_0\in\mathcal{O}$. We claim that for any $\Lambda_1\in\mathcal{O}$, the geodesic connecting $\Lambda_0$ and $\Lambda_1$ is length minimizing. This can be shown using the proof for the finite dimensional case given in \cite[Section~3.3]{docarmo}. In fact, it uses the following:
	\begin{enumerate}[label=(\alph*)]
		\item\label{dca} The connection is metric and symmetric.
		\item\label{dcb} The exponential map $\exp_{\Lambda_0}:A_{\Lambda_0}\to\mathcal{O}$ is bijective.
		\item\label{dcc} A family of functions $(h_t)\subset A_{\Lambda_0}$ is smooth if and only if the Lagrangian path $(\exp_{\Lambda_0}(h_t))$ is smooth.
	\end{enumerate}
	Condition \ref{dca} holds in our case, by \cite[Theorem~4.1]{solomon-curv}. The existence part of Theorem~\ref{boundary} implies that the exponential map is onto, and the uniqueness part of the same theorem implies it is one to one. Hence, condition \ref{dcb} holds. Condition \ref{dcc} also holds, as can be verified by following the proofs of Theorem~\ref{initial} and Proposition~\ref{solutionwhencylindrical}. Hence, geodesics in $\mathcal{O}$ are indeed length minimizing, and $(\mathcal{O},\overline{\Upsilon})$ is a metric space. By Lemma~\ref{triangle}, the exponential map is an isometry. The theorem follows.
\end{proof}

\subsection*{Acknowledgements}
The authors thank P. Giterman, Y. Rubinstein, P. Seidel and A. Solomon, for helpful conversations. This work was partially supported by ERC Starting Grant 337560 and BSF Grant 2012236.


\providecommand{\bysame}{\leavevmode\hbox to3em{\hrulefill}\thinspace}
\providecommand{\MR}{\relax\ifhmode\unskip\space\fi MR }
\providecommand{\MRhref}[2]{%
	\href{http://www.ams.org/mathscinet-getitem?mr=#1}{#2}
}
\providecommand{\href}[2]{#2}


\noindent
Institute of Mathematics \\
Hebrew University \\
jake@math.huji.ac.il \\

\noindent
Institute of Mathematics \\
Hebrew University \\
amitai.yuval@mail.huji.ac.il \\

\end{document}